\documentclass[11pt,a4paper,reqno]{amsart}
\usepackage[english]{babel}
\usepackage[applemac]{inputenc}
\usepackage[T1]{fontenc}
\usepackage{palatino}
\usepackage{amsmath}
\usepackage{amssymb}
\usepackage{amsthm}
\usepackage{amsfonts}
\usepackage{graphicx}
\usepackage{mathtools}

\usepackage[colorlinks = true, citecolor = black]{hyperref}
\author{Tuomas Orponen}
\title{An improved bound on the packing dimension of Furstenberg sets in the plane}
\address{University of Helsinki, Department of Mathematics and Statistics}
\thanks{T.O. is supported by the Academy of Finland through the grant Restricted families of projections and connections to
Kakeya type problems, grant number 274512.}
\email{tuomas.orponen@helsinki.fi}

\newcommand{\R}{\mathbb{R}}

\newcommand{\N}{\mathbb{N}}

\newcommand{\Z}{\mathbb{Z}}

\newcommand{\calT}{\mathcal{T}}
\newcommand{\calL}{\mathcal{L}}
\newcommand{\calD}{\mathcal{D}}
\newcommand{\calH}{\mathcal{H}}

\newcommand{\calB}{\mathcal{B}}

\newcommand{\calQ}{\mathcal{Q}}

\newcommand{\spt}{\operatorname{spt}}
\newcommand{\Hd}{\dim_{\mathrm{H}}}
\newcommand{\Pd}{\dim_{\mathrm{p}}}
\newcommand{\Bd}{\overline{\dim}\,}
\newcommand{\calP}{\mathcal{P}}

\newcommand{\diam}{\operatorname{diam}}

\newcommand{\dist}{\operatorname{dist}}

\numberwithin{equation}{section}

\theoremstyle{plain}
\newtheorem{thm}[equation]{Theorem}
\newtheorem{conjecture}[equation]{Conjecture}
\newtheorem{lemma}[equation]{Lemma}

\newtheorem{proposition}[equation]{Proposition}

\theoremstyle{definition}

\newtheorem{definition}[equation]{Definition}

\theoremstyle{remark}
\newtheorem{remark}[equation]{Remark}

\begin{document}

\begin{abstract} Let $0 \leq s \leq 1$. A set $K \subset \R^{2}$ is a Furstenberg $s$-set, if for every unit vector $e \in S^{1}$, some line $L_{e}$ parallel to $e$ satisfies
\begin{displaymath} \Hd [K \cap L_{e}] \geq s. \end{displaymath}
The Furstenberg set problem, introduced by T. Wolff in 1999, asks for the best lower bound for the dimension of Furstenberg $s$-sets. Wolff proved that $\Hd K \geq \max\{s + 1/2,2s\}$ and conjectured that $\Hd K \geq (1 + 3s)/2$. The only known improvement to Wolff's bound is due to Bourgain, who proved in 2003 that $\Hd K \geq 1 + \epsilon$ for Furstenberg $1/2$-sets $K$, where $\epsilon > 0$ is an absolute constant. In the present paper, I prove a similar $\epsilon$-improvement for all $1/2 < s < 1$, but only for packing dimension: $\Pd K \geq 2s + \epsilon$ for all Furstenberg $s$-sets $K \subset \R^{2}$, where $\epsilon > 0$ only depends on $s$.

The proof rests on a new incidence theorem for finite collections of planar points and tubes of width $\delta > 0$. As another corollary of this theorem, I obtain a small improvement for Kaufman's estimate from 1968 on the dimension of exceptional sets of orthogonal projections. Namely, I prove that if $K \subset \R^{2}$ is a linearly measurable set with positive length, and $1/2 < s < 1$, then
\begin{displaymath} \Hd \{e \in S^{1} : \Pd \pi_{e}(K) \leq s\} \leq s - \epsilon \end{displaymath}
for some $\epsilon > 0$ depending only on $s$. Here $\pi_{e}$ is the orthogonal projection onto the line spanned by $e$. 
\end{abstract} 

\maketitle

\tableofcontents

\section{Introduction}

This paper is concerned with two closely related topics in planar fractal geometry: the Furstenberg set problem, and exceptional sets of orthogonal projections. 
\subsection{Furstenberg sets} I start with the central definition:
\begin{definition}[Furstenberg sets]\label{furstenbergSets} Let $0 \leq s \leq 1$. A set $K \subset \R^{2}$ is a Furstenberg $s$-set, if there exists a set of unit vectors $S_{K} \subset S^{1}$ with positive length such that the following holds: for every $e \in S_{K}$, some line $L_{e}$ parallel to $e$ satisfies $\Hd [K \cap L_{e}] \geq s$. Here $\Hd$ is Hausdorff dimension. \end{definition}
The terminology was introduced in 1999 by T. Wolff \cite{Wo}, who proved the following lower bound for the Hausdorff dimension of Furstenberg sets:
\begin{thm}[Wolff's bound] The Hausdorff dimension of compact Furstenberg $s$-sets is at least $\max\{1/2 + s,2s\}$. \end{thm}
Wolff suspected that his bound is not sharp, and made the following conjecture:
\begin{conjecture}[Wolff's conjecture]\label{furstenbergConjecture} The Hausdorff dimension of compact Furstenberg $s$-sets is at least $(1 + 3s)/2$. \end{conjecture}
\begin{remark} Where does the name "Furstenberg set" come from? In 1970, H. Furstenberg \cite{Fu} proved the following theorem. Assume that $p,q \in \N$ are integers such that $\log p/\log q$ is irrational. Assume that $A_{p},A_{q} \subset [0,1]$ are closed sets invariant under $x \mapsto px (\operatorname{mod} 1)$ and $x \mapsto qx (\operatorname{mod} 1)$, respectively. Let $0 \leq s \leq 1$, and assume that some line $L \subset \R^{2}$ satisfies $\Hd [L \cap (A_{p} \times A_{q})] \geq s$. Then $A_{p} \times A_{q}$ is a Furstenberg $s$-set.  

Furstenberg was interested in the problem: how do lines intersect product sets of the form $A_{p} \times A_{q}$? He conjectured that 
\begin{equation}\label{FurstenbergsConjecture} \Hd [L \cap (A_{p} \times A_{q})] \leq \max\{0,\Hd (A_{p} \times A_{q}) - 1\} \end{equation}
for every line $L \subset \R^{2}$. Keeping in mind Furstenberg's theorem cited above, the upper bound \eqref{FurstenbergsConjecture} would evidently follow, if only one could prove the lower bound $\Hd K \geq s + 1$ for all Furstenberg $s$-sets. However, the estimate $\Hd K \geq s + 1$ is too optimistic in general: as shown by Wolff \cite{Wo}, Conjecture \ref{furstenbergConjecture} is the strongest possible for general Furstenberg sets. In other words, the sets in Definition \ref{furstenbergSets} are too general to help solve Furstenberg's conjecture \eqref{FurstenbergsConjecture} for the special sets of the form $A_{p} \times A_{q}$. 

Fortunately, the services of general Furstenberg sets are no longer required for this purpose: only this year, Furstenberg's conjecture \eqref{FurstenbergsConjecture} was spectacularly verified (independently) by P. Shmerkin \cite{Sh} and M. Wu \cite{Wu}, with two very different techniques! \end{remark}

For general sets, progress in Conjecture \ref{furstenbergConjecture} has been quite modest. Around the year 2000, Katz and Tao \cite{KT} observed that improving Wolff's bound for Furstenberg $\tfrac{1}{2}$-sets is roughly equivalent to proving a "$\delta$-discretised" sum-product theorem in $\R$. The latter task was then accomplished in 2003 by Bourgain \cite{Bo1}. Hence, the combined efforts of Katz-Tao and Bourgain give the following improvement to Wolff's bound:
\begin{thm}[Bourgain, Katz-Tao] There is an absolute constant $\epsilon > 0$ such that the Hausdorff dimension of compact Furstenberg $\tfrac{1}{2}$-sets is at least $1 + \epsilon$. \end{thm}
Of course, the theorem also gives an improvement for Furstenberg $s$-sets with $s$ very close to $\tfrac{1}{2}$, but, to the best of my knowledge, Wolff's bound remains the world record for other values of $s \in (0,1)$. The main purpose of this paper is to prove a Bourgain-Katz-Tao type $\epsilon$-improvement to Wolff's bound for all values $\tfrac{1}{2} < s < 1$. As a notable caveat, the method only works for packing dimension:
\begin{thm}\label{mainFurstenberg} For $\tfrac{1}{2} < s < 1$, there exists a constant $\epsilon = \epsilon(s) > 0$ such that every Furstenberg $s$-set $K$ has packing dimension $\Pd K \geq 2s + \epsilon$. \end{thm}
The foremost reason, why the proof of Theorem \ref{mainFurstenberg} does not give information about Hausdorff dimension -- or even lower Minkowski dimension -- is that it relies on the counter assumption $\Pd K \approx 2s$, which gives information about $K$ on two different scales, namely $\delta$ and $\delta^{1/2}$. Assuming $\Hd K \approx 2s$ does not have similar consequences. For a reader familiar with Besicovitch sets, I mention that a similar issue seems to stand in the way of improving Wolff's bound $\tfrac{5}{2}$ for the Hausdorff dimension of Besicovitch sets in $\R^{3}$: the improved lower bound $\tfrac{5}{2} + \epsilon$ from 2000 by Katz, \L aba and Tao \cite{KLT} is only known for upper Minkowski dimension. (Addendum to a second version of the paper: in April 2017, Katz and Zahl \cite{KZ} posted on arXiv a proof that the Hausdorff dimension of Besicovitch sets in $\R^{3}$ is at least $\tfrac{5}{2} + \epsilon$.)

Finally, I mention that several papers have been written around Wolff's conjecture \ref{furstenbergConjecture} in the past few years. An article of Zhang \cite{Zh} completely solves a discrete variant of the conjecture, plus its analogues in higher dimensions. Zhang also studied a variant of the problem in finite fields \cite{Zh2}. Most recently, Ellenberg and Erman \cite{EE} used machinery from algebraic geometry to study a "$k$-plane" variant of Wolff's conjecture in finite fields.

\subsection{Projections} The second topic of the paper are orthogonal projections. This is one of the most classical -- and popular -- topics in fractal geometry, so the amount of literature is immense: for a reader interested in finding out (much) more than covered below, I suggest taking a look at the recent survey of Fraser, Falconer and Jin \cite{FFJ}. 

Fix $0 \leq s \leq 1$, and let $K \subset \R^{2}$ be a Borel set of Hausdorff dimension $\Hd K \geq s$. In 1968, Kaufman \cite{Ka} proved, improving an earlier result of Marstrand \cite{Ma} from 1954, that
\begin{equation}\label{kaufman} \Hd \{e \in S^{1} : \Hd \pi_{e}(K) < s\} \leq s. \end{equation}
Here $\pi_{e} \colon \R^{2} \to \R$ is the orthogonal projection $\pi_{e}(x) = x \cdot e$. Under the assumption $\Hd K \geq s$, Kaufman's bound \eqref{kaufman} is sharp: in 1975, Kaufman and Mattila \cite{KM} constructed explicit compact sets $K \subset \R^{2}$ with $\Hd K = s$ such that
\begin{equation}\label{sharpness} \Hd \{e : \Hd \pi_{e}(K) < s\} = s. \end{equation}
Under the assumption $\Hd K \geq t > s$, the sharpness of \eqref{kaufman} is an open problem. The following improvement is conjectured (in (1.8) of \cite{Ob}, for instance):
\begin{conjecture}\label{mainC} Assume that $0 \leq t/2 \leq s \leq t \leq 1$ and $\Hd K \geq t$. Then
\begin{equation}\label{mainC2} \Hd \{e \in S^{1} : \Hd \pi_{e}(K) < s\} \leq 2s - t. \end{equation}
\end{conjecture} 

It is well-known that there is a connection between the case $t = 1$ of Conjecture \ref{mainC} and Wolff's conjecture \ref{furstenbergConjecture} for Furstenberg sets. As observed in 2012 by D. Oberlin \cite{Ob2}, an improvement to Conjecture \ref{mainC} immediately gives an improvement to Wolff's bound for Furstenberg sets arising from a special -- but rather natural -- construction. As far as I know, there is no published evidence of a converse, but it seems very likely that progress in Wolff's conjecture \ref{furstenbergConjecture} would also lead to progress in Conjecture \ref{mainC}. 

I now concentrate on the case $t = 1$. If $0 \leq s \leq 1$ and $K \subset \R^{2}$ is a Borel set with $\Hd K \geq 1$, then $\Hd K \geq s$, and \eqref{kaufman} holds for $K$. Curiously, it appears to be very difficult to capitalise on the stronger assumption $\Hd K \geq 1$, and beat the estimate \eqref{kaufman}. In fact, the only known improvement to Kaufman's bound \eqref{kaufman} follows -- once again -- from Bourgain's discretised sum-product theorem. The next theorem appeared in another paper of Bourgain \cite{Bo} from 2010:
\begin{thm}[Bourgain]\label{bourgain} Given $\kappa > 0$, there exists $\eta > \tfrac{1}{2}$ such that the following holds. If $K \subset \R^{2}$ is a Borel set with $\Hd K \geq 1$, then $\Hd \pi_{e}(K) \geq \eta$ for all $e \in S^{1} \setminus E$, where $E \subset S^{1}$ is an exceptional set of Hausdorff dimension $\Hd E \leq \kappa$. In particular,
\begin{displaymath} \Hd \{e \in S^{1} : \Hd \pi_{e}(K) \leq s\} \to 0, \quad \text{as } s \searrow \tfrac{1}{2}. \end{displaymath}
\end{thm}  
In brief, Theorem \ref{bourgain} marks a substantial improvement over Kaufman's bound \eqref{kaufman} for values of $s$ very close to $\tfrac{1}{2}$, but for other values of $s \in (1/2,1)$, Kaufman's bound remains the world record. It is no coincidence that the situation is reminiscent of the known bounds for Furstenberg $s$-sets, for $s$ close to, or far from, $\tfrac{1}{2}$. 

The second main result of the paper is a small improvement for the packing dimension variant of Kaufman's bound \eqref{kaufman}, for any $\tfrac{1}{2} < s < 1$:
\begin{thm}\label{mainProjections} Let $\tfrac{1}{2} < s < 1$. If $K \subset \R^{2}$ is an $\calH^{1}$-measurable set with $\calH^{1}(K) > 0$, then
\begin{displaymath} \Hd \{e \in S^{1} : \Pd \pi_{e}(K) \leq s\} \leq s - \epsilon \end{displaymath}
for some $\epsilon > 0$ depending only on $s$. \end{thm}
The reason for the appearance of $\Pd$ is the same as in Theorem \ref{mainFurstenberg}, and the proof does not to give any improvement for the dimension of $\{e : \Hd \pi_{e}(K) \leq s\}$. The assumption $\calH^{1}(K) > 0$ is quite convenient, but nothing more: the proof would also work for Borel sets $K$ with $\Hd K \geq 1$.

Theorem \ref{mainProjections} first appeared in a preliminary version of this paper \cite{O1} (which is now superseded by the current article, and hence not intended for publication). In the present paper, the proofs of Theorems \ref{mainFurstenberg} and \ref{mainProjections} are deduced from a single discrete result, Theorem \ref{mainIntro} below, which concerns incidences between certain finite families of points and $\delta$-tubes in the plane. At the level of this incidence result, Theorem \ref{mainProjections} is strictly easier than Theorem \ref{mainFurstenberg}, as the relevant families of $\delta$-tubes are somewhat special. 

\subsection{Outline of the paper} Both main results, Theorem \ref{mainFurstenberg} and \ref{mainProjections}, will be proven simultaneously. Section \ref{reductions} reduces the proofs to compact sets, and to corresponding claims about upper Minkowski dimension (instead of packing dimension). Section \ref{scaleSection} reduces the proofs further to the discrete result mentioned above, namely Theorem \ref{mainIntro}. Section \ref{incidenceSection} -- which is the main section of the paper -- contains the proof of Theorem \ref{mainIntro}. 

The reductions to Theorem \ref{mainIntro} are fairly standard, so Theorem \ref{mainIntro} can be considered the main result of the paper. It states, roughly, the following: if every point $p$ in a $\delta$-discretised $1$-dimensional set $P \subset \R^{2}$ is incident to a $\delta$-discretised $s$-dimensional set $\calT_{p}$ of $\delta$-tubes, then at least one of the following holds. Either $\calT = \cup_{p} \calT_{p}$ contains $\gg \delta^{-2s}$ tubes in total, or then it takes $\gg \delta^{-s}$ tubes of width $\delta^{1/2}$ to cover the union of the tubes in $\calT$.

The proof has two phases: the first -- and longer -- reduces the proof to sets $P$ of a special form, which I call "quasi-product sets" in lack of a better term. This phase is elementary but tedious. One starts with a counter assumption: the union of the tubes in $\calT$ can be covered by $\lesssim \delta^{-2s}$ and $\lesssim \delta^{-s}$ tubes at scales $\delta$ and $\delta^{1/2}$, respectively. Building on this information, one eventually finds a single $\delta^{1/2}$ tube $T_{0}$ with the following properties. First, $T_{0}$ contains a large number of points from $P$. Second, each point in $P \cap T_{0}$ is incident to a large number of $\delta$-tubes $T$, which are essentially contained in $T_{0}$ (in particular, this can be used to find an upper bound on the total number of relevant tubes $T$). After such a $\delta^{1/2}$-tube $T_{0}$ has been found, one applies an affine re-scaling $A$, which essentially sends $P \cap T_{0}$ inside the unit square, and maps the $\delta$-tubes $T$ to $\delta^{1/2}$-tubes, see Figure \ref{fig2}. Then, it turns out that $A(P \cap T_{0})$ behaves like a quasi-product set, and has suspiciously many incidences with the $\delta^{1/2}$-tubes $A(T)$. 
\begin{figure}[h!]
\begin{center}
\includegraphics[scale = 0.4]{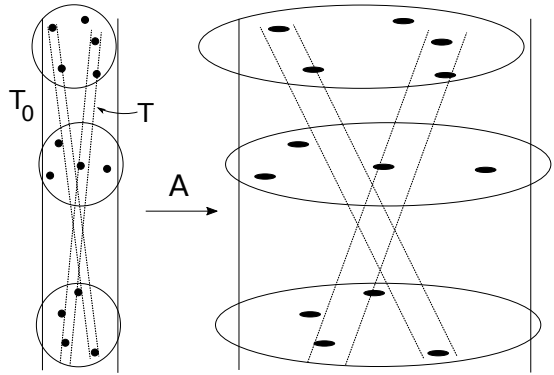}
\caption{The tubes $T_{0},T$ and the set $P \cap T_{0}$, before and after the affine transformation $A$. Explaining why $A(P \cap T_{0})$ "behaves like a quasi-product set" would get too technical here, so I refer to Section \ref{quasiProduct} for more details.}\label{fig2}
\end{center}
\end{figure}

At this point, it may seem like all the work has been fruitless: apart from changing scales from $\delta$ to $\delta^{1/2}$, the original incidence problem associated to $P$ and $\calT$ has precisely the same numerology as the new incidence problem associated to $A(P \cap T_{0})$ and the $\delta^{1/2}$-tubes $A(T)$. However, it turns out that the problem is easier to solve (or at least make progress in) for quasi-product sets, because tools from additive combinatorics become available. 

In the second phase, one proves an incidence theorem for quasi-product sets (see Proposition \ref{productProp}). This uses standard tools from additive combinatorics, such as the Pl\"unnecke-Ruzsa inequalities and the Balog-Szemer\'edi-Gowers theorem. In the end, it turns out that the incidence problem for quasi-product sets is roughly equivalent to a discretised variant of Bourgain's projection theorem, Theorem \ref{bourgain}. Fortunately, Bourgain states and proves a suitable discretised variant of Theorem \ref{bourgain} in his paper \cite{Bo}, so the proof of Theorem \ref{mainIntro} is completed by appealing to Theorem 5 in \cite{Bo}. 

\subsection{Some notation} An open ball in $\R^{d}$ with centre $x$ and radius $r > 0$ will be denoted by $B(x,r)$. The Hausdorff measure and content of dimension $t$ will be denoted by $\calH^{t}$ and $\calH^{t}_{\infty}$, respectively. Given real numbers $A,B > 0$, the notation $A \lesssim B$ means that $A \leq CB$ for some constant $C \geq 1$. If the dependence of $C$ on a parameter $p$ needs to be emphasised, I will write $A \lesssim_{p} B$. The notation $A \gtrsim B$ means that $B \lesssim A$, and $A \sim B$ stands for $A \lesssim B \lesssim A$. 

The notation $A \lesssim_{\log} B$ means that $A \lesssim \log^{C}(1/\delta) B$ for some absolute constant $C \geq 1$, where $\delta > 0$ is a "scale". The meaning of $\delta > 0$ will be clear from the context, whenever the notation is used. The notations $A \gtrsim_{\log} B$ and $A \sim_{\log} B$ are then defined as above. 

Given a bounded set $F \subset \R^{d}$, the notation $N(F,\delta)$ stands for the least number of balls of radius $\delta$ required to cover $F$. The \emph{upper Minkowski dimension} of $F$ is
\begin{displaymath} \Bd F := \limsup_{\delta \to 0} \frac{\log N(F,\delta)}{-\log \delta}. \end{displaymath}
The \emph{packing dimension} $\Pd$ is defined in \eqref{packingDimension} below. 

\subsection{Acknowledgements} I wish to thank the referees for reading the manuscript carefully and giving excellent comments; they helped me make the paper more readable.


\section{Reductions to Minkowski dimension and compact sets}\label{reductions} In this short section, I reduce the proofs of Theorems \ref{mainFurstenberg} and \ref{mainProjections} to establishing analogous statements for Minkowski dimension (instead of packing dimension), and just for compact sets. 

I start with reductions concerning Furstenberg sets. Let $K \subset \R^{2}$ be an arbitrary Furstenberg $s$-set, and let $S_{K} \subset S^{1}$ be the associated set of unit vectors with $\calH^{1}(S_{K}) > 0$. The definition of packing dimension is
\begin{equation}\label{packingDimension} \Pd K = \inf \left\{\sup_{i} \Bd F_{i} :  K \subset \bigcup_{i} F_{i} \right\}, \end{equation}
where the $\inf$ is taken over all countable covers of $K$ with bounded sets $F_{i}$. Since taking closures does not affect the upper Minkowski dimension $\Bd$, one may restrict attention to covers by compact sets $F_{i}$. Now, given any $\epsilon > 0$, I claim that one of the sets $F_{i}$ is (essentially) a Furstenberg $(s - \epsilon)$-set. This is rather straightforward: since the sets $F_{i}$ cover $K$, one has
\begin{displaymath} \Hd [K \cap L_{e}] = \sup_{i} \Hd [F_{i} \cap L_{e}], \qquad e \in S_{K}, \end{displaymath}
so for any fixed $e \in S^{1}$, it holds that $\calH^{s - \epsilon}_{\infty}(F_{i} \cap L_{e}) > 0$ for some $i$. Consequently, there exists $i$ such that $\calH^{s - \epsilon}_{\infty}(F_{i} \cap L_{e}) \geq c > 0$ for some $c > 0$ and for a positive set of vectors $e$. Thus, $F_{i}$ is a Furstenberg $(s - \epsilon)$-set. Since $\epsilon > 0$ was arbitrary, it follows that it suffices to prove Theorem \ref{mainFurstenberg} for the Minkowski dimension $\Bd$, for compact sets $K$, and under the extra assumption that
\begin{equation}\label{contentAss} \calH^{s}_{\infty}(K \cap L_{e}) \geq c > 0, \qquad e \in S_{K}. \end{equation}

As a slightly less obvious reduction, I claim that, without loss of generality, one may assume that the lines $L_{e}$, $e \in S_{K}$, form a compact set. To formalise the statement, I recall the (standard) concept of point-line duality in the plane:
\begin{definition}[Point-line duality]\label{pointLineDuality} The points in $\R^{2}$ are in one-to-one correspondence with non-vertical lines in $\R^{2}$ via the mapping
\begin{displaymath} \calD : (a,b) \mapsto \{y = ax + b : x \in \R\}. \end{displaymath}
For every set of points $P$, define the set of lines $\calL_{P} := \calD(P) = \{\calD(p) : p \in P\}$. Similarly, for a set $\calL$ of non-vertical lines $\calL$, define the set of points $P_{\calL} := \calD^{-1}(\calL)$. A family of lines will be called open/closed/compact etc. if the point set $P_{\calL}$ has the same topological property. I will also write $\calH^{t}(\calL) := \calH^{t}(P_{\calL})$. 
\end{definition}

\begin{remark}\label{notationalRemark} Even though the family of lines $\calD(P)$ and the planar set $\cup\{L : L \in \calD(P)\}$ are different objects, I will not differentiate between them in subsequent notation. In particular, if $B \subset \R^{2}$ is any set, the notation $B \cap \calD(P)$ refers to $\cup \{B \cap L : L \in \calD(P)\}$.  \end{remark}

Now, let $\calL_{K} := \{L_{e} : e \in S_{K}\}$. Deleting a set of lines with sufficiently small measure, one may assume that every line in $\calL_{K}$ makes a positive (and uniformly bounded from below) angle with the $y$-axis.  Then $P_{K} := P_{\calL_{K}}$ is a bounded graph, that is, a set of the form $\{(a,f(a)) : a \in A\}$, where $A \subset \R$ is a bounded set of positive length, and $f$ is a bounded function. Consider the compact line set
\begin{displaymath} \overline{\calL}_{K} := \calL_{\overline{P}_{K}}. \end{displaymath} 
Then every line $L \in \overline{\calL}_{K}$ has the property \eqref{contentAss}. This follows easily from the upper semi-continuity of Hausdorff content with respect to Hausdorff convergence and the compactness of $K$. Namely, assume for a moment that $L \in \overline{\calL}_{K}$ is such that $K \cap L_{e}$ can be covered by finitely many open balls $B_{i}$ with 
\begin{displaymath} \sum_{i} d(B_{i})^{s} < c. \end{displaymath}
Choose a sequence of lines $L_{j} \in \calL_{K}$ converging to $L$ locally in the Hausdorff metric. Then, using the compactness of $K$, the sets $K \cap L_{j}$ can also be covered by the balls $B_{i}$ for all $j$ large enough, contradicting \eqref{contentAss}. 

The set $\overline{P}_{K}$ may no longer be a graph, but it certainly satisfies $\calH^{1}(\overline{P}_{K}) > 0$. 
\begin{definition}[Generalised Furstenberg $s$-set]\label{genFS} Let $0 < s < 1$. Assume that $K \subset B(0,1) \subset \R^{2}$ is a compact set, and $\calL$ is a compact set of lines with $\calH^{1}(\calL) > 0$ with the property that $\calH^{s}_{\infty}(K \cap L) \geq c > 0$ for every $L \in \calL$, and for some constant $c > 0$. Also, assume that every line $L \in \calL$ makes an angle $\geq 1/10$ with the $y$-axis. Then $K$ is called a \emph{generalised Furstenberg $s$-set}. \end{definition} 
By the previous discussion (and a simple coordinate-change, if necessary, to accommodate the angle requirement), the proof of Theorem \ref{mainFurstenberg} is now reduced to proving to following statement:
\begin{thm}\label{mainGeneralised} $\Bd K \geq 2s + \epsilon$ for every generalised Furstenberg $s$-set $K$, where $\epsilon > 0$ only depends on $s$. \end{thm}

I now turn to the -- much shorter -- reduction related to the projection result, Theorem \ref{mainProjections}. The following observation is a special case Lemma 4.5 in \cite{O2}:
\begin{lemma} Assume that $K \subset \R^{2}$ is $\calH^{1}$-measurable with $\calH^{1}(K) > 0$, and 
\begin{displaymath} \Hd \{e \in S^{1} : \Pd \pi_{e}(K) < \sigma\} > \beta \end{displaymath}
for some $\sigma,\beta > 0$. Then, there exists a compact set $K' \subset K$ with $\calH^{1}(K') > 0$ such that $\calH^{1}(K') > 0$ and
\begin{displaymath} \Hd \{e \in S^{1} : \Bd \pi_{e}(K') < \sigma\} > \beta. \end{displaymath}
\end{lemma}

It follows immediately that it suffices to prove Theorem \ref{mainProjections} for $\Bd$ instead of $\Pd$, and for compact sets $K$ with $\calH^{1}(K) > 0$. Also, one may restrict attention to the case $0 < \calH^{1}(K) < \infty$, since compact subsets with finite measure can always be found, and proving the theorem for any subset implies it for the whole set. For purposes of easy reference, I record the result explicitly:
\begin{thm}\label{mainProjections2} Let $1/2 < s < 1$. If $K \subset \R^{2}$ is a is compact with $0 < \calH^{1}(K) < \infty$, then
\begin{displaymath} \Hd \{e \in S^{1} : \Bd \pi_{e}(K) \leq s\} \leq s - \epsilon \end{displaymath}
for some $\epsilon > 0$ depending only on $s$. \end{thm}

\section{Proofs of Theorems \ref{mainGeneralised} and \ref{mainProjections2}}\label{scaleSection}

The proofs of both the main theorems are based on counter assumptions. If Theorem \ref{mainGeneralised} fails, then for arbitrarily small $\epsilon > 0$, there exists a generalised Furstenberg $s$-set $K_{F} \subset \R^{2}$ such that 
\begin{equation}\label{KF} N(K_{F},\delta) \leq \delta^{-2s - \epsilon} \end{equation}
for all small enough $\delta > 0$, say $0 < \delta \leq \delta_{1}$. Recall that $N(K,\delta)$ is the least number of balls of radius $\delta$ required to cover $K$. Similarly, if Theorem \ref{mainProjections2} fails, then for arbitrarily small $\epsilon > 0$, there exists a number $1/2 < s < 1$, a compact set $K_{\pi} \subset \R^{2}$ with $0 < \calH^{1}(K_{\pi}) < \infty$, and a set of vectors $E \subset S^{1}$ with $\calH^{s}(E) > 0$ such that
\begin{equation}\label{KPi} N(\pi_{e}(K_{\pi}),\delta) \leq \delta^{-s - \epsilon}, \qquad e \in E, \end{equation}
for all $0 < \delta \leq \delta_{2}$. 

The purpose of this section is to first pick the scale $0 < \delta \leq \delta_{0} \leq \min\{\delta_{1},\delta_{2}\}$ so that the information from the counter assumptions \eqref{KF} and \eqref{KPi} is as strong as possible. Then, at this scale $\delta > 0$, the counter assumptions are employed to construct an "impossible" configuration of $\delta$-separated points and $\delta$-tubes. The "impossibility" of the configuration is finally deduced from an incidence result, Theorem \ref{mainIntro}. The proof of the incidence result is a separate story, which will occupy the remainder of the paper.

\subsection{Finding the scale $\delta > 0$} The scale $\delta$ will only be chosen once, but it can be chosen arbitrarily small (by choosing $\delta_{0} \leq \min\{\delta_{1},\delta_{2}\}$ very small). This will be useful -- and implicitly assumed -- countless times below. For example, I will always implicitly assume that $\delta^{-\epsilon}$ is far larger than various constants $C \geq 1$, which appear throughout the proof.

I start by recalling some basic facts about "discretising $s$-dimensional sets". As far as I know, the following definition is due to Katz and Tao \cite{KT}: 

\begin{definition}[$(\delta,s,C)$-sets]\label{deltaSSet} Fix $\delta,s > 0$. A finite $\delta$-separated set $P \subset \R^{d}$ is called a $(\delta,s,C)$-set, if
\begin{equation}\label{deltaTSet} |P \cap B(x,r)| \leq C\left(\frac{r}{\delta} \right)^{s} \end{equation}
for all $x \in \R^{d}$ and $\delta \leq r \leq 1$. Here and below, $|\cdot |$ stands for cardinality. 
\end{definition}

An open ball of radius $\delta > 0$ will be called a \emph{$\delta$-ball}. A collection of $\delta$-balls will be called a $(\delta,s,C)$-set, if the centres of the balls form a $(\delta,s,C)$-set. The following proposition explains the rationale behind $(\delta,s,C)$-sets:

\begin{proposition}\label{deltasSet} Let $\delta > 0$, and let $B \subset B(0,1) \subset \R^{2}$ be a set with $\calH^{s}_{\infty}(B) =: \kappa > 0$. Then, there exists a $(\delta,s,C)$-set $P \subset B$ with cardinality $|P| \geq (\kappa/C) \cdot \delta^{-s}$, where $C \geq 1$ is an absolute constant. \end{proposition}
The proof is very close to that Frostman's lemma; the details can be found in the appendix of \cite{FO}. Note that any $(\delta,s,C)$-set $P \subset B(0,1)$ satisfies $|P| \leq C\delta^{-s}$. So, slightly informally, Proposition \ref{deltaSSet} says that sets $B \subset B(0,1)$ with $\calH^{s}(B) > 0$ contain $(\delta,s,C)$-sets with near-maximal cardinality.

Now, I will pick a suitable scale $\delta > 0$. Recall the sets $K_{F}$ and $K_{\pi}$. Since $K_{F}$ is a generalised Furstenberg $s$-set, it comes bundled with a compact set of lines $\calL$ with $\calH^{1}(\calL) > 0$. Clearly, one may also assume that $\calH^{1}(\calL) < \infty$, since a compact finite-measure subset $\calL' \subset \calL$ can be found, and then the pair $K_{F}, \calL'$ satisfies the same hypotheses as $K_{F},\calL$. 

Let $K$ be either one of the sets $K_{\pi}$ or $P_{\calL}$, so that $K \subset B(0,1)$ and $0 < \calH^{1}(K) < \infty$. I will treat $\calH^{1}(K)$ as an absolute constant in the $\lesssim$-notation below; in particular $\calH^{1}(K) \sim 1$. Let $\mu$ be a Frostman measure supported on $K$, that is, $\mu(K) = 1$ and $\mu(B(x,r)) \lesssim r$ for all balls $B(x,r) \subset \R^{2}$. Next, let $\calB$ be an efficient $\delta_{0}$-cover for $K$, that is,
\begin{equation}\label{form190} \sup\{\diam B : B \in \calB\} \leq \delta_{0} \quad \text{and} \quad \sum_{B \in \calB} \diam(B) \lesssim \calH^{1}(K) \sim 1. \end{equation} 
One may assume that the diameters of the balls in $\calB$ are of the form $2^{-j}$, $j \in \N$. For $2^{-j} \leq \delta_{0}$, set $\calB_{j} := \{B \in \calB : \diam(B) = 2^{-j}\}$, and observe that
\begin{displaymath} \sum_{2^{-j} \leq \delta_{0}} \sum_{B \in \calB_{j}} \mu(B) \geq \mu(K) = 1. \end{displaymath}
In particular, there exists an index $j \in \N$ with $2^{-j} \leq \delta_{0}$ and
\begin{equation}\label{form117} \sum_{B \in \calB_{j}} \mu(B) \gtrsim \frac{1}{(j - j_{0} + 1)^{2}}. \end{equation}
Here $j_{0} \in \N$ is the smallest number with $2^{-j_{0}} \leq \delta_{0}$. Now, I set
\begin{displaymath} \delta := 2^{-2j}, \end{displaymath}
so that $\delta^{1/2} = 2^{-j}$. Note that $\delta^{1/2} \leq \delta_{0}$. In particular, \eqref{form117} implies that
\begin{equation}\label{form180} \sum_{B \in \calB_{j}} \mu(B) \gtrsim_{\log} 1. \end{equation}
Observe that $|\calB_{j}| \lesssim \delta^{-1/2}$ by \eqref{form190}, and on the other hand every ball $B \in \calB_{j}$ satisfies $\mu(B) \lesssim \delta^{1/2}$. Thus, \eqref{form180} implies that there are $\sim_{\log} \delta^{-1/2}$ balls in $\calB_{j}$, denoted by $\calB_{j}^{G}$, such that
\begin{equation}\label{form200} \mu(B) \gtrsim_{\log} \delta^{1/2}, \qquad B \in \calB_{j}^{G}. \end{equation}
Discarding a few balls if necessary, one may assume that 
\begin{equation}\label{form210} \dist(B,B') \geq \delta^{-1/2}, \qquad B,B' \in \calB_{j}^{G}. \end{equation}
For each ball $B \in B_{j}^{G}$, choose a $(\delta,1,C)$-set $P_{B} \subset B$ with $C \sim 1$ and $|P_{B}| \gtrsim_{\log} \delta^{-1/2}$. This is possible by Proposition \ref{deltasSet}, since \eqref{form200} and the linear growth of $\mu$ imply that $\calH_{\infty}^{1}(B \cap K) \gtrsim_{\log} \delta^{1/2}$. Write
\begin{displaymath} P := \bigcup_{B \in \calB_{j}^{G}} P_{B}. \end{displaymath}
Then $|P| \sim_{\log} \delta^{-1}$, and $P_{B} = B \cap P$ for $B \in \calB_{j}^{G}$. I will now verify that $P$ is a $(\delta,1,C)$-set for some $C \sim_{\log} 1$. Fix $x \in P$ and $r \geq \delta$, let $B \in \calB_{j}^{G}$ be the unique ball with $x \in P_{B}$. There are two cases to consider: if $\delta \leq r \leq \delta^{1/2}$, one needs only note that $|P \cap B(x,r)| = |P_{B} \cap B(x,r)|$ by \eqref{form210}, and recall that $P_{B}$ is a $(\delta,1,C)$-set with $C \sim 1$. So, let $r \geq \delta^{1/2}$. This time, if $B(x,r) \cap B' \neq \emptyset$ for some ball $B' \in \calB_{j}^{G}$, then $r$ is large enough to ensure that $B' \subset B(x,2r)$. Hence, by $|P \cap B'| \lesssim \delta^{-1/2}$ for $B' \in \calB_{j}^{G}$, and \eqref{form200}, and the disjointness of the balls in $\calB_{j}^{G}$, one obtains
\begin{displaymath} |P \cap B(x,r)| \lesssim \delta^{-1/2} \mathop{\sum_{B' \in \calB_{j}^{G}}}_{B' \cap B(x,r) \neq \emptyset} \frac{\mu(B')}{\mu(B')} \lesssim_{\log} \delta^{-1} \mathop{\sum_{B' \in \calB_{j}^{G}}}_{B' \subset B(x,2r)} \mu(B') \leq \frac{\mu(B(x,2r))}{\delta} \lesssim \frac{r}{\delta}. \end{displaymath} 

I recap the main achievements so far. For a suitable scale $\delta \leq \delta_{0}^{2}$, a $(\delta,1,C)$-set $P \subset K \in \{K_{\pi},P_{\calL}\}$ has now been constructed with $C \sim_{\log} 1$, along with a family of $\delta^{1/2}$-balls $\calB$ such that
\begin{itemize}
\item[(P1)]\label{P1} $|P| \sim_{\log} \delta^{-1}$, and $|\calB| \lesssim \delta^{-1/2}$,
\item[(P2)] $P \subset \bigcup_{B \in \calB} B$.
\end{itemize}

\subsection{Finding $\delta$-tubes} Next, relying on the counter assumptions \eqref{KF} and \eqref{KPi}, I will accompany $P$ with a finite family of $\delta$-tubes. For technical reasons, I will consider two types of $\delta$-tubes in this paper: the \emph{ordinary} ones, which are $(\delta/2)$-neighbourhoods of lines in $\R^{2}$, and then the \emph{dyadic} ones, which I now proceed to define:


\begin{definition}[Dyadic tubes]\label{dyadicTubes} For $\delta = 2^{-k}$, $k \geq 0$, a dyadic $\delta$-tube is a set of the form $\calD(Q)$, where $Q \subset [0,1)^{2}$ is a dyadic square of side-length $\delta$, and $\calD$ is the point-line duality mapping from Definition \ref{pointLineDuality}. One should view $\calD(Q)$ here as a set of points in $\R^{2}$, not as a family of lines, see Remark \ref{notationalRemark}. The \emph{slope} of a dyadic $\delta$-tube $T = \calD([a + \delta) \times [b + \delta))$ is defined fo be $s(T) := a$ (which is the actual slope of the line $\calD(a,b) \subset T$). 

The definition above is convenient for the reason that dyadic tubes have a dyadic structure (unlike ordinary tubes). For dyadic numbers $0 < \delta_{1} < \delta_{2} < 1$, the \emph{$\delta_{2}$-parent} of a $\delta_{1}$-tube $T_{1} = \calD(Q_{1})$ is the unique tube $T_{2} = \calD(Q_{2})$ such that $Q_{2}$ is a dyadic square of side-length $\delta_{2}$ containing $Q_{1}$. The \emph{$\delta_{1}$-children} of a $\delta_{2}$-tube are defined in the obvious way: note that a $\delta_{1}$-tube $T_{1}$ is the child of a $\delta_{2}$-tube $T_{2}$, if and only if $T_{1} \subset T_{2}$. Given a collection of $\delta_{1}$-tubes $\calT$, I write $N(\calT,\delta_{2})$ for the cardinality of the family of $\delta_{2}$-parents of the tubes in $\calT$ (that is, minimal family of $\delta_{2}$-tubes containing all the tubes in $\calT$).

From geometric intents and purposes, dyadic tubes are locally very similar to ordinary tubes: if $T$ is a dyadic $\delta$-tube, then the intersection $T \cap B(0,R)$ is contained in an ordinary $C_{R}\delta$-tube with the same slope, for some constant $C_{R} \geq 1$ depending only on $R$. Conversely, an ordinary $\delta$-tube can be covered by a bounded number of dyadic $\delta$-tubes with nearly the same slope (with an error of $\lesssim \delta$).

Assume that $\calT_{p}$ is a collection of dyadic $\delta$-tubes, each containing a point $p \in \R^{2}$, and let $0 < s < 1$. Then $\calT_{p}$ is called a $(\delta,s,C)$-set, if the set of slopes $s(\calT_{p}) := \{s(T) : T \in \calT_{p}\}$ is a $(\delta,s,C)$-subset of $\R$. The same definition is used, if $p$ is a $\delta$-ball instead of a point, and "containing $p$" is replaced by "intersecting $p$".

Similarly, a family of ordinary $\delta$-tubes, all containing a common point, is called a $(\delta,s,C)$-set, if the directions of the tubes (on $S^{1}$) form a $(\delta,s,C)$-set. \end{definition}

As stated above Definition \ref{dyadicTubes}, the plan is to use the counter assumptions \eqref{KF} and \eqref{KPi} to accompany $P$ with a finite family of dyadic $\delta$-tubes $\calT$. Finding $\calT$ is a somewhat lengthy task, so I start by clarifying: what exactly is required of these tubes to end up with a contradiction? In brief, the tubes need violate the next theorem:

\begin{thm}\label{mainIntro} Given $0 < s < 1$, there exists an $\epsilon = \epsilon(s) > 0$ such that the following holds for small enough dyadic numbers $\delta > 0$ (depending only on $s$). Assume that $P \subset B(0,1) \subset \R^{2}$ is a $(\delta,1,\delta^{-\epsilon})$-set with cardinality $|P| \geq \delta^{-1 + \epsilon}$ and assume that
\begin{equation}\label{deltaHalfAssumptionIntro} N(P,\delta^{1/2}) \leq \delta^{-1/2 - \epsilon}. \end{equation}
Assume that $\calT$ is a collection of dyadic $\delta$-tubes such that for every $p \in P$, there exists a sub-family $\calT_{p} \subset \{T \in \calT : p \in T\}$, which is a $(\delta,s,\delta^{-\epsilon})$-set of cardinality $|\calT_{p}| \geq \delta^{-s + \epsilon}$. Then either
\begin{equation}\label{alternativeIntro} |\calT| \geq \delta^{-2s - \epsilon} \quad \text{or} \quad N(\calT,\delta^{1/2}) \geq \delta^{-s - \epsilon}. \end{equation}
\end{thm}

Note that the assumptions of Theorem \ref{mainIntro}, in particular \eqref{deltaHalfAssumptionIntro}, are valid for the set $P$ constructed earlier, by properties (P1)--(P2). So, to prove Theorems \ref{mainGeneralised} and \ref{mainProjections2}, it remains to use \eqref{KF} and \eqref{KPi} to find a family of dyadic $\delta$-tubes $\calT$ which violates Theorem \ref{mainIntro}. I first need to record a few easy geometric lemmas about points and dyadic $\delta$-tubes:

\begin{lemma}\label{tubesAndSlopes} Assume that $\calT_{p}$ is a collection of dyadic $\delta$-tubes, each containing a point $p \in B(0,1)$. Then $|s(\calT_{p})| \sim |\calT_{p}|$. \end{lemma}
\begin{proof} Write $S := s(\calT_{p})$. Clearly $|S| \leq |\calT_{p}|$, so it suffices to prove that $|\calT_{p}| \lesssim |S|$. To this end, I will show that only four tubes in $\calT_{p}$ can share a common slope. Assume that $a \in S$, and $T_{1} = \calD([a,a + \delta) \times [b_{1},b_{1} + \delta))$ and $T_{2} = \calD([a,a + \delta) \times [b_{2},b_{2} + \delta))$ both belong to $\calT_{p}$, so that $p = (p_{x},p_{y}) \in T_{1} \cap T_{2}$. By definition of $T_{1},T_{2}$, this means that there exist numbers $a',b_{1}',a'',b_{2}''$ with $\max\{|a' - a|,|a'' - a|,|b_{1}' - b_{1}|,|b_{2}'' - b_{2}|\} < \delta$ such that
\begin{displaymath} p_{y} = a'p_{x} + b_{1}' \quad \text{and} \quad p_{y} = a''p_{x} = b_{2}''. \end{displaymath}
It follows that $|b_{1} - b_{2}| \leq |p_{x}||a' - a''| + 2\delta < 4\delta$. This completes the proof. 
  \end{proof}  

\begin{lemma}\label{auxLemma} Assume that $0 < \delta_{1} < \delta_{2}$ are dyadic numbers, $p \in B(0,1)$, and $T_{0} = \calD(Q_{0}) = \calD([a,a + \delta_{2}) \times [b,b + \delta_{2}))$ is a dyadic $\delta_{2}$-tube. Further, assume that $\calT_{p}$ is a $(\delta_{1},s,C)$-set of dyadic $\delta_{1}$-tubes $T$ with $p \in T \subset T_{0}$. Then $|\calT_{p}| \lesssim C(\delta_{2}/\delta_{1})^{s}$.
\end{lemma}

\begin{proof} Let $\calQ_{p}$ be the collection of dyadic $\delta_{1}$-squares such that $\calT_{p} = \{\calD(Q) : Q \in \calQ_{p}\}$. Then $Q \subset Q_{0}$, $Q \in \calQ_{p}$, by the assumption $T \subset T_{0}$, $T \in \calT_{p}$. Hence $s(\calT_{p})$ is a $(\delta_{1},s,C)$-subset of $[a,a + \delta_{2})$, and consequently $|s(\calT_{p})| \leq C(\delta_{2}/\delta_{1})^{s}$. The previous lemma completes the proof.\end{proof} 
 
 \begin{lemma}\label{tubeCovering} Assume that $P \subset B(0,1)$ is a set and $0 < \delta_{1} \leq \delta_{2}$ are dyadic numbers. Assume that $a_{2} \in \delta_{2}\Z$, and $P$ can be covered by $M \in \N$ dyadic $\delta_{2}$-tubes $\calD([a_{2},a_{2} + \delta_{2}) \times [b_{j},b_{j} + \delta_{2}))$ with fixed slope $a_{2}$. Then the collection of all dyadic $\delta_{1}$-tubes, which intersect $P$ and have slope in $[a_{2},a_{2} + \delta_{2})$, can be covered by $\lesssim M$ dyadic $\delta_{2}$-tubes with slope $a_{2}$. 
 \end{lemma}
 
 \begin{proof} Assume that $\calD([a_{1},a_{1} + \delta_{1}) \times [b_{1},b_{1} + \delta_{1}))$ is a $\delta_{1}$-tube with slope $a_{1} \in [a_{2},a_{2} + \delta_{2})$, which intersects $P$ at a point $p = (p_{x},p_{y})$. This means that $p_{y} = a_{1}'p_{x} + b_{1}'$ for some $|a_{1}' - a_{1}| < \delta_{1}$, and consequently $|a_{1}' - a_{2}| \leq 2\delta_{2}$. By assumption, $p$ is also covered by one of the dyadic $\delta_{2}$-tubes $\calD([a_{2},a_{2} + \delta_{2}) \times [b_{j},b_{j} + \delta_{2}))$, $1 \leq j \leq M$, which implies that $p_{y} = a_{2}'p_{x} + b_{j}'$ for some $|a_{2}' - a_{2}| < \delta_{2}$ and $|b_{j}' - b_{j}| < \delta_{2}$. It follows that
 \begin{displaymath} |b_{1}' - b_{j}'| = |p_{x}||a_{1}' - a_{2}'| \leq 3\delta_{2}, \end{displaymath}
 and consequently $|b_{1} - b_{j}| \leq 5\delta_{2}$. Now, the $\delta_{2}$-tubes of the form $\calD([a_{2},a_{2} + \delta) \times [b,b + \delta_{2}))$ with $|b - b_{j}| \leq 5\delta$ for some $1 \leq j \leq M$, form the desired cover.
 \end{proof}
  
 
Now, finally, everything is set up to accompany $P$ with a family of dyadic $\delta$-tubes $\calT$. The process is slightly different in the cases $K = K_{\pi}$ and $K = P_{\calL}$. 
\subsubsection{The case $P \subset K = K_{\pi}$} Here we use the counter assumption \eqref{KPi}, restated below: 
\begin{equation}\label{KPi2} N(\pi_{e}(K_{\pi}),\delta) \leq \delta^{-s - \epsilon}, \quad e \in E, \: 0 < \delta \leq \delta_{0}, \end{equation}
where $\calH^{s}(E) > 0$. The plan is to find a $(\delta,s)$-subset $E' \subset E$ of near-maximal cardinality and use \eqref{KPi2} to cover $P \subset K_{\pi}$ by a small family of dyadic $\delta$-tubes $\calT_{e}$ (nearly) perpendicular to $e$ for every $e \in E'$. 

Recall that both $\delta$ and $\delta^{1/2}$ are dyadic numbers. Let $E_{\delta} \subset \delta \Z$ be a $(\delta,s,C)$-set of slopes almost perpendicular to the vectors in $E$: more precisely, if $\theta \in E_{\delta}$, the requirement is that any line with slope $\theta$ is perpendicular to some vector in $E(10\delta)$ (the $10\delta$-neighbourhood is only needed to facilitate $E_{\delta} \subset \delta \Z$). By Proposition \ref{deltaSSet} applied to $E$ and some easy tinkering, one can choose $E_{\delta}$ with $|E_{\delta}| \sim \delta^{-s}$ (the implicit constants naturally depend on $\calH^{s}_{\infty}(E) > 0$, which will be treated as an absolute constant). Inequality \eqref{KPi2} then implies that for every $\theta \in E_{\delta}$, the set $K_{\pi}$ can be covered by $\lesssim \delta^{-s - \epsilon}$ dyadic $\delta$-tubes $\calT_{\theta}$ with slope $\theta$. Without loss of generality, all tubes will be assumed to intersect $K_{\pi}$. Let
\begin{displaymath} \calT := \bigcup_{\theta \in E_{\delta}} \calT_{\theta}. \end{displaymath}
Then $\calT$ is a collection of dyadic $\delta$-tubes such that
\begin{itemize}
\item[(T1)] $|\calT| \lesssim \delta^{-2s - \epsilon}$,
\item[(T2)] for every point $p \in P \subset K$, there is a $(\delta,s,C)$-subset $\calT_{p} \subset \{T \in \calT : p \in T\}$ of cardinality $|\calT_{p}| \sim \delta^{-s}$.
\item[(T3)] $N(\calT,\delta^{1/2}) \lesssim \delta^{-s - \epsilon}$.
\end{itemize}
Here (T1) follows from $|E_{\delta}| \sim \delta^{-s}$ and $|\calT_{\theta}| \lesssim \delta^{-s - \epsilon}$. To see the (T2), fix $p \in P \subset K = K_{\pi}$ and observe that $p \in T$ for some tube $T \in \calT_{\theta}$, for every $\theta \in E_{\delta}$. The family $\calT_{p}$ can be picked among those tubes $T$. The the claim (T3) requires a bit of extra work. I first recall the following basic estimate:
\begin{proposition}\label{kaufmanProp} Assume that $K \subset B(0,1)$ with $\calH^{1}(K) > 0$ and $0 < t < 1$. Then
 \begin{displaymath} N(\{e \in S^{1} : N(\pi_{e}(K),\delta) \leq \delta^{-t}\},\delta) \lesssim_{\log} \delta^{-t}, \end{displaymath}
 where the implicit constants depend on $\calH^{1}(K)$. 
 \end{proposition}
 This is, for instance, inequality (1.2) in \cite{O}, and the proof can be found on p. 9 of the same paper. The proposition also easily follows from Lemma \ref{2sBound} below. Now, apply Proposition \ref{kaufmanProp} with $t = s + \epsilon$, at scale $\delta^{1/2}$, and the set $K_{\pi}$. The conclusion is that
 \begin{equation}\label{form30} N(E_{\delta},\delta^{1/2}) \lesssim N(\{e \in S^{1} : N(\pi_{e}(K_{\pi}),\delta^{1/2}) \leq C\delta^{-(s + \epsilon)/2}\},\delta^{1/2}) \lesssim_{\log} \delta^{-(s + \epsilon)/2}. \end{equation}
The first inequality follows from the fact that if $\theta \in E_{\delta}(\delta^{1/2})$, then a unit vector $e_{\theta}$ parallel to a line with slope $\theta$ lies at distance $\lesssim \delta^{1/2}$ from some vector $e \in E^{\perp}$. Then, if $e_{\theta}^{\perp},e^{\perp}$ are perpendicular to $e_{\theta},e$, one has $|e_{\theta}^{\perp} - e^{\perp}| \lesssim \delta^{1/2}$, and
\begin{equation}\label{form31} N(\pi_{e_{\theta}^{\perp}}(K_{\pi}),\delta^{1/2}) \lesssim N(\pi_{e^{\perp}}(K_{\pi}),\delta^{1/2}) \leq \delta^{-(s + \epsilon)/2} \end{equation}
by the definition of $E$ (and $\delta^{1/2} \leq \delta_{0}$). This means that $e_{\theta}^{\perp}$ belongs to the set in the middle of \eqref{form30} for large enough $C \geq 1$, which implies the first inequality of \eqref{form30} for the vectors $e_{\theta}$, $\theta \in E_{\delta}$ -- and then for the points $\theta \in E_{\delta}$. 

With \eqref{form30} in hand, pick a collection $\{\theta_{j}\}_{j \in \mathcal{J}} \subset \delta^{1/2}\Z$ such that $|\mathcal{J}| \lesssim \delta^{-(s + \epsilon)/2}$ and $E_{\delta}$ is contained in the union of the dyadic intervals $[\theta_{j},\theta_{j} + \delta^{1/2})$. The numbers $\theta_{j}$ can be picked at distance $\leq \delta^{1/2}$ from some point in $E_{\delta}$, so \eqref{form31} applies with $\theta = \theta_{j}$: the conclusion is that $K_{\pi}$ can be covered by $\lesssim \delta^{-(s + \epsilon)/2}$ dyadic $\delta^{1/2}$-tubes with slope $\theta_{j}$. Consequently, by Lemma \ref{tubeCovering}, all the tubes in the families $\calT_{\theta}$, with $\theta \in E_{\delta} \cap [\theta_{j},\theta_{j} + \delta^{1/2})$, can be covered by $\lesssim \delta^{-(s + \epsilon)/2}$ dyadic $\delta^{1/2}$-tubes of slope $\theta_{j}$. It follows that all the tubes in $\calT$ can be covered by $\lesssim \delta^{-(s + \epsilon)/2} \cdot \delta^{-(s + \epsilon)/2} = \delta^{-s - \epsilon}$ dyadic $\delta^{1/2}$-tubes, as claimed in (T3). 

\subsubsection{The case $P \subset K = P_{\calL}$} This case is simpler: again, the aim is to find a family of dyadic $\delta$-tubes $\calT$ satisfying the conditions (T1)--(T3). Recall the main counter assumption \eqref{KF}:
\begin{equation}\label{KF+} N(K_{F},\delta) \leq \delta^{-2s - \epsilon}, \qquad 0 < \delta \leq \delta_{0}. \end{equation}
Also, recall that $\calL$ is a compact set of lines $L$, which form a large angle with the $y$-axis, and with the property that $\calH^{s}_{\infty}(K_{F} \cap L) \geq c > 0$. By (P1), $P \subset P_{\calL}$ is a $(\delta,1,C)$-set with $C \sim_{\log} 1$ of cardinality $|P| \sim_{\log} \delta^{-1}$. 

I record a small observation about the point-line duality: 
\begin{lemma}\label{dualityLemma} Assume that $(c,d) \in \calD(a,b)$. Then $(a,b) \in \calD(-c,d)$. \end{lemma} 
\begin{proof} By assumption $d = ac + b$, or $b = (-c)a + d$. Hence $(a,b) \in \calD(-c,d)$. \end{proof}

Note that $L_{p} := \calD(p) \in \calL$ for all $p \in P \subset P_{\calL}$. It follows from $\calH^{s}_{\infty}(K_{F} \cap L_{p}) \geq c$ and Proposition \ref{deltaSSet} that $K_{F} \cap L_{p}$ contains a $(\delta,s,C)$-set with $C \sim 1$ and cardinality $\sim \delta^{-s}$ (I treat $c$ as an absolute constant). Let $\calQ_{p}' = \{(a_{i},a_{i} +  \delta] \times [b_{i}, b_{i} + \delta)\}_{i \in \mathcal{I}}$ be the collection of (not quite dyadic) $\delta$-squares of the form $(a, a + \delta] \times [b,b + \delta)$ with $a,b \in \delta \Z$, which contain a point in the said $(\delta,s,C)$-set on $K_{F} \cap L_{p}$. Then $\calQ_{p} := \{[-a_{i} - \delta,a_{i}) \times [b_{i}, b_{i} + \delta)\}_{i \in \mathcal{I}}$ is a collection of dyadic squares. Since $L_{p}$ is quantitatively non-vertical, the numbers $a_{i}$ form a $(\delta,s,C)$-set, and hence $\calT_{p} := \{\calD(Q) : Q \in \calQ_{p}\}$ is a $(\delta,s,C)$-set of dyadic $\delta$-tubes by definition. Each tube $\calD([-a_{i} - \delta,a_{i}) \times [b_{i},b_{i} + \delta)) \in \calT_{p}$ moreover contains $p$, since $(a_{i},a_{i} + \delta] \times [b_{i},b_{i} + \delta)$ contains a point $(x_{1},x_{2}) \in L_{p} = \calD(p)$ by definition, and then 
\begin{displaymath} p \in \calD(-x_{1},x_{2}) \in \calD([-a_{i} - \delta,a_{i}) \times [b_{i},b_{i} + \delta)) \end{displaymath}
by Lemma \ref{dualityLemma}. Writing
\begin{displaymath} \calT := \bigcup_{p \in P} \calT_{p}, \end{displaymath}
the condition (T2) is automatically valid. Note that $\calT$ consists of $\calD$-images of $\delta$-squares $Q$ meeting $\{(-x,y) : (x,y) \in K_{F}\}$. Hence $|\calT| \lesssim \delta^{-2s - \epsilon}$ by \eqref{KF+}, which gives (T1). Finally, all these $\delta$-squares can be covered by $\lesssim \delta^{-s - \epsilon/2}$ dyadic squares of side-length $\delta^{1/2}$, by \eqref{KF+} at scale $\delta^{1/2}$. This gives $N(\calT,\delta^{1/2}) \lesssim \delta^{-s - \epsilon/2}$, which is a little bit better than (T3). 

\subsection{Concluding the proofs of Theorems \ref{mainGeneralised} and \ref{mainProjections2}}\label{contradictionSection} In both cases $K = K_{\pi}$ and $K = \calP_{\calL}$, a finite set $P$ and a collection of tubes $\calT$ have now been found, satisfying (P1)--(P2) and (T1)--(T3), respectively. The counter assumptions \eqref{KF} and \eqref{KPi} were heavily used. Now, to conclude the proofs of both Theorems \ref{mainGeneralised} and \ref{mainProjections2}, it suffices to show that $P$ can $\calT$ cannot exist, for sufficiently small $\epsilon > 0$. This follows from Theorem \ref{mainIntro}. Indeed, we already observed earlier that $P$ satisfies the hypotheses of Theorem \ref{mainIntro}, by (P1)--(P2). Moreover, the condition (T2) is even slightly better than what Theorem \ref{mainIntro} requires from $\calT$, and conditions (T1) and (T3) literally state that the main conclusion \eqref{alternativeIntro} of Theorem \ref{mainIntro} fails (assuming that $\epsilon > 0$ is sufficiently small). So, a contradiction has been reached, and the proofs of Theorems \ref{mainGeneralised} and \ref{mainProjections2} are complete.

\section{An incidence bound for points and tubes}\label{incidenceSection}

The purpose of this section is to prove Theorem \ref{mainIntro}, stated below as Theorem \ref{main} for convenience. I start with a simple -- and well-known -- incidence bound: heuristically, Theorem \ref{main} can then be viewed as an $\epsilon$-improvement of this "trivial" bound, although the hypotheses are somewhat stronger.

 \begin{lemma}\label{2sBound} Let $0 < s < 1$ and assume that $\epsilon > 0$ is small enough (depending on $s$ only). Assume that $P \subset B(0,1) \subset \R^{2}$ is a $(\delta,1,\delta^{-\epsilon})$-set with $|P| \geq \delta^{-1 + \epsilon}$, and $\calT$ is a family of dyadic $\delta$-tubes. Assume that for every point $p \in P$, there exists a sub-family $\calT_{p} \subset \{T \in \calT : p \in T\}$ with $|\calT_{p}| \geq \delta^{-s + \epsilon}$. Then $|\calT| \gtrsim_{\log} \delta^{-2s + 6\epsilon}$. \end{lemma}
 
 \begin{proof} Define the set of \emph{incidences} as follows:
 \begin{displaymath} I(P,\calT) := \{(p,T) \in P \times \calT : T \in \calT_{p}\}. \end{displaymath}
 Note the slightly non-standard definition: the condition $p \in T$ is necessary, but not sufficient, for $(p,T) \in I(P,\calT)$. Evidently
 \begin{displaymath} |I(P,\calT)| \geq \sum_{p \in P} |\calT_{p}| \geq \delta^{-1 - s + 2\epsilon}. \end{displaymath}
 Write
 \begin{displaymath} N_{T} := |\{p \in P : T \in \calT_{p}\}|, \end{displaymath}
 and estimate $|I(P,\calT)|$ from above as follows:
 \begin{align*} |I(P,\calT)| & = \sum_{T \in \calT} N_{T} \leq |\calT|^{1/2} \left( \sum_{T \in \calT} |\{(p,q) \in P \times P : T \in \calT_{p} \cap \calT_{q}\}| \right)^{1/2}\\
 & \lesssim |\calT|^{1/2} \left( \sum_{p \in P} |\calT_{p}| \right)^{1/2} + |\calT|^{1/2} \left( \sum_{p \neq q} |\calT_{p} \cap \calT_{q}| \right)^{1/2}\end{align*} 
 The first term equals $|\calT|^{1/2}|I(P,\calT)|^{1/2}$. So, in case the first term dominates, one obtains $|\calT| \gtrsim |I(P,\calT)| \geq \delta^{-1 - s - 2\epsilon}$, which beats the desired estimate, if $\epsilon > 0$ is small enough. To estimate the second sum, one observes that
 \begin{displaymath} |\calT_{p} \cap \calT_{q}| \leq |\{T \in \calT_{p} : q \in T\}| \lesssim \frac{1}{|p - q|}, \end{displaymath}
 using the fact that the slopes of the tubes in $\calT_{p}$ are $\delta$-separated (first prove the inequality for any family of ordinary $\delta$-tubes, which contain $p$ and have $\delta$-separated slopes, and finally use $P \subset B(0,1)$ to reduce the dyadic case the non-dyadic one). Hence, if the second sum dominates, one obtains
 \begin{displaymath} |I(P,\calT)| \lesssim |\calT|^{1/2} \left( \sum_{p \neq q} \frac{1}{|p - q|} \right)^{1/2} \lesssim_{\log} |\calT|^{1/2} \delta^{-1 - \epsilon}. \end{displaymath}
The second inequality is a standard estimate using the $(\delta,1,\delta^{-\epsilon})$-set hypothesis (for each point $p \in P$, divide the points $q \neq p$ into dyadic annuli around $p$, and make the obvious estimates). Consequently $|\calT| \gtrsim_{\log} \delta^{-2s + 6\epsilon}$, as claimed. 
 \end{proof}
 
 \begin{remark}\label{ballsNotPoints} The lemma works verbatim the same, if the points $p$ are replaced by disjoint $\delta$-balls $B$, and $\calT_{B}$ (instead of $\calT_{p}$) consists of dyadic $\delta$-tubes meeting $B$ (instead of containing $p$). Instead of $|p - q|$, consider the distance between the centres of the relevant balls. \end{remark}
 
The next theorem is just Theorem \ref{mainIntro} repeated for convenience:
 
\begin{thm}\label{main} Given $0 < s < 1$, there exists an $\epsilon = \epsilon(s) > 0$ such that the following holds for small enough dyadic numbers $\delta > 0$ (depending only on $s$). Assume that $P \subset B(0,1) \subset \R^{2}$ is a $(\delta,1,\delta^{-\epsilon})$-set with cardinality $|P| \geq \delta^{-1 + \epsilon}$ and assume that
\begin{equation}\label{deltaHalfAssumption} N(P,\delta^{1/2}) \leq \delta^{-1/2 - \epsilon}. \end{equation}
Assume that $\calT$ is a collection of dyadic $\delta$-tubes such that for every $p \in P$, there exists a sub-family $\calT_{p} \subset \{T \in \calT : p \in T\}$, which is a $(\delta,s,\delta^{-\epsilon})$-set of cardinality $|\calT_{p}| \geq \delta^{-s + \epsilon}$. Then either
\begin{equation}\label{alternative} |\calT| \geq \delta^{-2s - \epsilon} \quad \text{or} \quad N(\calT,\delta^{1/2}) \geq \delta^{-s - \epsilon}. \end{equation}
\end{thm}

\begin{remark}\label{calTcalP} Without loss of generality, one may clearly assume that
\begin{displaymath} \calT = \bigcup_{p \in P} \calT_{p}. \end{displaymath}
However, it is good to keep in mind that $p \in T \in \calT$ can, nevertheless, happen for some tubes $T \notin \calT_{p}$. It would make life somewhat easier, if one could assume 
\begin{displaymath} ``T \in \calT_{p} \Longleftrightarrow p \in T", \end{displaymath}
but I do not know how to make such a reduction. \end{remark}

\begin{remark} I suspect that the assumptions of Theorem \ref{main} are unnecessarily strong. The following conjecture seems plausible. Let $0 < s < 1$ and $s \leq \tau < 1$. Assume that $P$ is a $(\delta,1,\delta^{-\epsilon})$-set with $|P| \geq \delta^{-1 + \epsilon}$, and assume that $\calT$ is a family of (dyadic) $\delta$-tubes such that, for every $p \in P$, the sub-family $\{T \in \calT : p \in T\}$ contains a $(\delta,\tau,\delta^{-\epsilon})$-set $\calT_{p}$ of cardinality $|\calT_{p}| \geq \delta^{-s + \epsilon}$. Then $|\calT| \geq \delta^{-2s - \epsilon}$, if $\epsilon > 0$ is small enough in a manner depending only on $s$ and $\tau$. This would give an improvement for the lower Minkowski dimension -- and possibly even Hausdorff dimension -- of Furstenberg sets. \end{remark} 

\subsection{The main counter assumption}\label{mainCounterAss} The proof of Theorem \ref{main} begins, and I make a counter assumption:
\begin{equation}\label{counterAss} \delta^{-2s + 6\epsilon} \lesssim_{\log} |\calT| \leq \delta^{-2s - \epsilon} \quad \text{and} \quad N(\calT,\delta^{1/2}) \leq \delta^{-s - \epsilon}. \end{equation}
(The lower bound for $|\calT|$ is simply a consequence of Lemma \ref{2sBound}.) In the sequel, I will constantly use the notations $A \lessapprox B$, $A \gtrapprox B$ and $A \approx B$ to signify equivalence up to a factor of $C_{\epsilon}\delta^{-C\epsilon}$, where $\epsilon > 0$ is the counter assumption parameter from \eqref{counterAss}. So, for instance $A \lessapprox B$ means that $A \leq C_{\epsilon}\delta^{-C\epsilon}B$ for some constant $C \geq 1$ depending only on $s$, and some constant $C_{\epsilon} \geq 1$ depending only on $\epsilon$ and $s$. It is also convenient to define that a finite collection of points or $\delta$-balls is a $(\delta,t)$-set, if it is a $(\delta,t,C_{\epsilon}\delta^{-C\epsilon})$-set for constants $C,C_{\epsilon} \geq 1$ as above.

In brief, the proof below shows that, under the counter assumption \eqref{counterAss}, certain quantities $A,B$ satisfy $A \lessapprox B$. However, Proposition \ref{productProp} below states that $A \geq \delta^{-\epsilon_{s}}B$ for some $\epsilon_{s} > 0$ depending only on $s$. Consequently, the counter assumption cannot hold for arbitrarily small $\epsilon > 0$, and Theorem \ref{main} follows.

Before starting in earnest, I gather a list of notation which will be introduced more carefully during the proof:
\begin{itemize}
\item[$B,\calB$] The letter $B$ stands for a ball of radius $\delta^{1/2}$, and $\calB$ stands for a family of $\delta^{1/2}$-balls, see the start of Section \ref{scaleDeltaHalf}.
\item[$P$] The letter $P$ stands for the $\delta$-separated set from the hypothesis of Theorem \ref{main}.
\item[$\calT,\calT_{\delta^{1/2}}$] The letter $\calT$ stands for the family of $\delta$-tubes from Theorem \ref{main}, whereas $\calT_{\delta^{1/2}}$ is a cover of the tubes in $\calT$ by $\delta^{1/2}$-tubes; see the definition above \eqref{form3}. 
\item[$'$] Apostrophes generally mean refinements: $P',\calT',\calB'$ are large subsets of $P,\calT,\calB$.
\item[$\calT_{p},\calT^{B}_{\delta^{1/2}}$] Subsets of $\calT$ or $\calT_{\delta^{1/2}}$ containing a fixed point $p$, or intersecting a fixed ball $B$, are denoted by $\calT_{p}$ or $\calT^{B}_{\delta^{1/2}}$; see the line after \eqref{form3}, and the hypothesis of Theorem \ref{main}. 
\item[$P_{B},p_{B},P_{\delta^{1/2}}$] A certain subset of $P \cap B$ is denoted by $P_{B}$. The notation $p_{B}$ stands for a point in $P \cap B$ (not necessarily in $P_{B}$), which has been "singled-out". For both the definitions, see the lines after \eqref{form5}. The union of the points $p_{B}$, over all the balls $B \in \calB$, is denoted by $P_{\delta^{1/2}}$, see \eqref{PDeltaHalf}.
\item[$M_{T},N_{T}$] The letter $M_{T}$ stands for the number of balls $B \in \calB$ intersecting the $\delta^{1/2}$-tubes $T$. Similarly, $N_{T}$ stands for the number of points in $P$ contained in a $\delta$-tube $T$. See the lines after \eqref{form7} and Lemma \ref{monotonicity} for definitions. 
\item[$\calT_{0},\calB_{0}$] The letter $\calT_{0}$ stands for a sub-family of $\calT$, consisting of $\delta$-tubes contained in a fixed $\delta^{1/2}$-tube $T_{0}$. The family $\calB_{0}$ consists of balls in $\calB$ meeting the same tube $T_{0}$. See \eqref{calT0} and below.
\end{itemize}

\subsection{Considerations at scale $\delta^{1/2}$}\label{scaleDeltaHalf} Let $\calB$ be a a collection of $\delta^{1/2}$-balls covering $P$. Then $|\calB| \lessapprox \delta^{-1/2}$ by the assumption \eqref{deltaHalfAssumption}, but clearly also $|\calB| \gtrapprox \delta^{-1/2}$, since $P$ is a $(\delta,1)$-set with $|P| \approx \delta^{-1}$. Moreover, since $|P \cap B| \lessapprox \delta^{-1/2}$ for every $B \in \calB$ by the $(\delta,1)$-set assumption, and $|P| \approx \delta^{-1}$, one sees that a subset $P' \subset P$ with $|P'| \approx \delta^{-1}$ is covered by balls $B \in \calB$ with
\begin{equation}\label{form2} |P \cap B| \approx \delta^{-1/2}. \end{equation}
Since $P'$ satisfies all the same assumptions as $P$ -- with slightly worse constants perhaps -- I may and will assume that $P' = P$; thus, one may assume that \eqref{form2} holds for all balls $B \in \calB$. By throwing away an additional fraction of the points in $P$, one can assume that the balls in $\calB$ are $\delta^{1/2}$-separated:
\begin{equation}\label{distB} \dist(B,B') \geq \delta^{1/2}, \qquad B,B' \in \calB. \end{equation}
(In the most relevant application of Theorem \ref{main}, to the set $P$ from the previous section, both \eqref{form2} and \eqref{distB} are \emph{a priori} guaranteed by \eqref{form210} and the construction of $P$ below \eqref{form210}). Furthermore, \eqref{form2} implies that $\calB$ is a $(\delta^{1/2},1)$-set: if the midpoints of the balls $B \in \calB$ are temporarily denoted by $R$, then
\begin{displaymath} \frac{r}{\delta} \gtrapprox |P \cap B(x,2r)| \gtrapprox |R \cap B(x,r)| \cdot \delta^{-1/2}, \qquad x \in \R^{2}, \: r \geq \delta^{1/2}. \end{displaymath}

Now, let $\calT_{\delta^{1/2}}$ be a collection of $\delta^{1/2}$-parents of the tubes in $\calT$. By the counter assumption \eqref{counterAss},
\begin{equation}\label{form3} |\calT_{\delta^{1/2}}| \lessapprox \delta^{-s}. \end{equation}
For $B \in \calB$, let $\calT_{\delta^{1/2}}^{B}$ be the collection of tubes in $\calT_{\delta^{1/2}}$ intersecting $B$. Pick a large constant $C \geq 1$. I now claim that at most half of the balls $B \in \calB$ can satisfy 
\begin{displaymath} |\calT_{\delta^{1/2}}^{B}| \geq \delta^{-s/2 - C\epsilon}. \end{displaymath}
Indeed, since the collection of balls $\calB$ is a $(\delta^{1/2},1)$-set, Lemma \ref{2sBound} applies at scale $\delta^{1/2}$ (see also Remark \ref{ballsNotPoints}). The conclusion is that if $|\calB|/2 \approx \delta^{-1/2}$ balls in $\calB$ did satisfy $|\calT_{\delta^{1/2}}^{B}| \geq \delta^{-s/2 - C\epsilon}$ the inequality above, then $|\calT_{\delta^{1/2}}| \gtrapprox \delta^{-s - C\epsilon + 6\epsilon}$.  For $C \geq 1$ large enough, this would contradict \eqref{form3}. 

I now discard all the balls from $\calB$ with $|\calT_{\delta^{1/2}}^{B}| \geq \delta^{-s/2 - C\epsilon}$ along with the points of $P$ contained in them. Since $|\calB|/2 \approx \delta^{-1/2}$ balls remain, and each of these balls satisfies \eqref{form2}, also $\approx \delta^{-1}$ points of $P$ remain. Thus, passing to these subsets of $P$ and $\calB$ if necessary, one may assume without loss of generality the uniform bound
\begin{equation}\label{secondRed} |\calT_{\delta^{1/2}}^{B}| \lessapprox \delta^{-s/2}, \qquad B \in \calB. \end{equation}
On the other hand, the tubes in $\calT_{\delta^{1/2}}^{B}$ cover all the tubes $T \in \calT_{p}$, for any individual $p \in B$. Since $\calT_{p}$ is a $(\delta,s)$-set, any fixed dyadic $\delta^{1/2}$-tube can only cover $\lessapprox \delta^{-s/2}$ tubes $T \in \calT_{p}$ by Lemma \ref{auxLemma}. Since $|\calT_{p}| \approx \delta^{-s}$ by assumption, it follows that $|\calT_{\delta^{1/2}}^{B}| \gtrapprox \delta^{-s/2}$. Combining this with \eqref{secondRed}, one obtains
\begin{equation}\label{form4} |\calT_{\delta^{1/2}}^{B}| \approx \delta^{-s/2}, \qquad B \in \calB. \end{equation}
By the argument in the paragraph above \eqref{secondRed}, this implies that
\begin{equation}\label{form3+} |\calT_{\delta^{1/2}}| \approx \delta^{-s}, \end{equation}
matching the upper bound in the counter assumption \eqref{counterAss}. I still need to regularise the situation a little further: even though \eqref{form4} now holds uniformly for $B \in \calB$, it can happen that the tubes $T \in \calT_{\delta^{1/2}}^{B}$ contain significantly different numbers tubes $T \in \calT_{p}$ -- and even worse, these numbers can depend on the choice of $p \in B \cap P$.  

To remedy this, fix $B \in \calB$ and any $p \in B \cap P$ for the moment. Then the tubes $T \in \calT_{p}$, are covered by the tubes in $\calT_{\delta^{1/2}}^{B}$ (since $p \in T \subset T_{0} \in \calT_{\delta^{1/2}}$ forces $T_{0} \cap B \neq \emptyset$). As already observed above, by Lemma \ref{auxLemma}, every tube $T \in \calT_{\delta^{1/2}}^{B}$ can only have $\lessapprox \delta^{-s/2}$ children in $\calT_{p}$. Since $|\calT_{p}| \approx \delta^{-s}$, it follows from this and \eqref{form4} that there necessarily exists a family 
\begin{equation}\label{form22} \calT_{\delta^{1/2}}^{B}(p) \subset \calT_{\delta^{1/2}}^{B} \end{equation}
with $|\calT_{\delta^{1/2}}^{B}| \approx \delta^{-s/2}$ such that every tube in $\calT_{\delta^{1/2}}^{B}(p)$ has $\approx \delta^{-s/2}$ children in $\calT_{p}$. Then the family $\calT_{\delta^{1/2}}^{B}(p)$ is a $(\delta^{1/2},s)$-set. To see this, fix a ball $B(x,r) \subset \R$ with $r \geq \delta^{1/2}$. If $B(x,r)$ contains $s(T_{0})$ for some $T_{0} \in \calT_{\delta^{1/2}}^{B}(p)$, then $B(x,2r)$ contains $s(T)$ for each of the $\approx \delta^{-s/2}$ tubes $T \in \calT_{p}$ with $T \subset T_{0}$. There are $\approx \delta^{-s/2}$ such slopes $s(T)$ by Lemma \ref{tubesAndSlopes}. Since $s(\calT_{p})$ is a $(\delta,s)$-set, $B(x,r)$ contains no more than $\lessapprox (r/\delta)^{s}$ elements in $s(\calT_{p})$, and consequently no more than $\lessapprox (r/\delta^{1/2})^{s}$ elements in $s(\calT_{\delta^{1/2}}^{B}(p))$. 

For every $p \in B \cap P$, the family $\calT_{\delta^{1/2}}^{B}(p) \subset \calT_{\delta^{1/2}}^{B}$ is a subset of cardinality 
\begin{displaymath} |\calT_{\delta^{1/2}}^{B}(p)| \approx |\calT_{\delta^{1/2}}^{B}| \approx \delta^{-s/2} \end{displaymath}
by \eqref{form5}, and so
\begin{displaymath} \sum_{T \in \calT_{\delta^{1/2}}^{B}} |\{p : T \in \calT_{\delta^{1/2}}^{B}(p)\}| = \sum_{p \in B \cap P} |\calT_{\delta^{1/2}}^{B}(p)| \gtrapprox |B \cap P||\calT_{\delta^{1/2}}^{B}|. \end{displaymath} 
Applying Cauchy-Schwarz on the left hand side then gives
\begin{displaymath} \sum_{p,q \in B \cap P} |\calT_{\delta^{1/2}}^{B}(p) \cap \calT_{\delta^{1/2}}^{B}(q)| \gtrapprox |B \cap P|^{2}|\calT_{\delta^{1/2}}^{B}|, \end{displaymath}
which implies that there exist $\approx |B \cap P|^{2} \approx \delta^{-1}$ pairs of points $p,q \in B \cap P$ such that
\begin{equation}\label{form5} |\calT_{\delta^{1/2}}^{B}(p) \cap \calT_{\delta^{1/2}}^{B}(q)| \approx |\calT_{\delta^{1/2}}^{B}| \approx \delta^{-s/2}. \end{equation}
Consequently, one can fix a single point $p = p_{B} \in P \cap B$ such that \eqref{form5} holds for $\approx \delta^{-1/2}$ points $q \in B \cap P$. Denote these points by $P_{B}$, so \eqref{form5} becomes
\begin{equation}\label{form5+} |\calT_{\delta^{1/2}}^{B}(p_{B}) \cap \calT_{\delta^{1/2}}^{B}(q)| \approx \delta^{-s/2}, \qquad q \in P_{B}. \end{equation}
Now, write
\begin{displaymath} \calT_{B}'' := \calT_{\delta^{1/2}}^{B}(p_{B}), \end{displaymath}
which is a $(\delta^{1/2},s)$-set of dyadic $\delta^{1/2}$-tubes meeting $B$ of cardinality $|\calT_{B}''| \approx \delta^{-s/2}$. In fact 
\begin{equation}\label{pBAndT} p_{B} \in T \text{ for all } T \in \calT_{B}'', \end{equation}
since the tubes in $\calT_{B}''$ contain some tubes in $\calT_{p_{B}}$. There are two apostrophes in $\calT_{B}''$, because the collection will, eventually, undergo two refinements (or "removals of exceptional sets"), and the end product will be denoted by $\calT_{B}$.

\subsection{Refining the families of $\delta^{1/2}$-tubes} For a tube $T = \calD([a,\delta^{1/2}) \times [b,\delta^{1/2})) \in \calT_{B}''$, let $e_{T} \in S^{1}$ for the unit vector perpendicular to line $\calD(a,b) \subset T$, and let $E_{B} := \{e_{T} : T \in \calT_{B}\}$. Then $E_{B}$ is a $(\delta^{1/2},s)$-set. Denote by $\pi_{e_{T}}$ the orthogonal projection onto the line spanned by $e_{T}$, see Figure \ref{fig3}.
\begin{figure}
\begin{center}
\includegraphics[scale = 0.4]{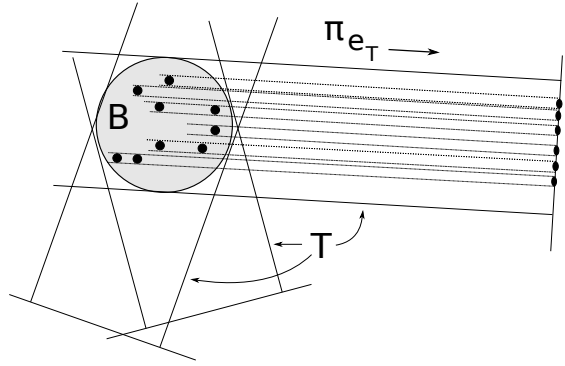}
\caption{The tubes $T \in \calT_{B}''$, and the projections $\pi_{e_{T}}$.}\label{fig3}
\end{center}
\end{figure}
Informally, the next lemma says that "for any $B \in \calB$, in an overwhelming majority of directions $e_{T} \in E_{B}$, the set $\pi_{e_{T}}(P_{B})$, and all its reasonably large subsets, contain a $(\delta,s)$-set of nearly maximal cardinality, namely $\approx \delta^{-s/2}$". 

\begin{lemma}\label{projections} Let $C_{0},C_{1},C_{2} \geq 1$ be constants. Then, if $C_{2}$ is sufficiently large, depending on $C_{0},C_{1}$ and the various constants behind the $\lessapprox$-notation used above, there are at least $(1 - \delta^{C_{1}\epsilon})|\calT_{B}''|$ "good" vectors $e_{T} \in E_{B}$ with the following property: if $P_{B}' \subset P_{B}$ is a subset of cardinality $|P_{B}'| \geq \delta^{C_{0}\epsilon}|P_{B}|$, then $\pi_{e_{T}}(P_{B}')$ contains a $(\delta,s,C_{2})$-set of cardinality $\geq \delta^{C_{2}\epsilon - s/2}$. \end{lemma}

\begin{proof} The proof is a variation of the standard "potential theoretic" argument, invented by Kaufman \cite{Ka}; if the reader is not familiar with the technique, a similar but cleaner statement is Theorem 2 in \cite{FM}. First, observe that $\delta^{-1/2}P_{B}$ is a $(\delta^{1/2},1)$-set of cardinality $\approx \delta^{-1/2}$. Next, consider the measures
\begin{displaymath} \mu_{B} := \frac{1}{|P_{B}|} \sum_{p \in \delta^{-1/2}P_{B}} \frac{\chi_{B(p,\delta^{1/2})}}{\delta} \end{displaymath}
and
\begin{displaymath} \nu := \frac{1}{|E_{B}|} \sum_{e \in E_{B}} \frac{\chi_{B(e,\delta^{1/2}) \cap S^{1}}}{\delta^{1/2}}, \end{displaymath}
and note that $\mu_{B}(\R^{2}) \sim 1 \sim \nu(S^{1})$. For $r \geq \delta^{1/2}$, one has the uniform estimates $\mu_{B}(B(x,r)) \lessapprox r$ and $\nu(B(e,r)) \lessapprox r^{s}$, while for $0 < r \leq \delta^{1/2}$ one has the obvious improved estimates. After some straightforward computations, it follows that
\begin{equation}\label{form29} \int_{S^{1}} I_{s}(\pi_{e\sharp}\mu_{B}) \, d\nu e := \iint \left[ \int_{S^{1}} \frac{d\nu e}{|\pi_{e}(x) - \pi_{e}(y)|^{s}} \right] \, d\mu_{B} x \, d\mu_{B} y \lessapprox 1. \end{equation}
Indeed, the inner integral (in brackets) can be estimated by $\lessapprox 1/|x - y|^{s}$, and then
\begin{displaymath} \int_{S^{1}} I_{s}(\pi_{e\sharp}\mu_{B}) \, d\nu e \lessapprox \int \left[\int \frac{d\mu_{B} x}{|x - y|^{s}} \right] \, d\mu_{B} y \lessapprox 1, \end{displaymath}
since the inner integral is again bounded by $\lessapprox 1$ for any $y \in \R^{2}$. Consequently, by Chebyshev's inequality, 
\begin{displaymath} \nu(\{e \in S^{1} : I_{s}(\pi_{e\sharp}\mu_{B}) \geq \delta^{-C\epsilon}\}) \lessapprox \delta^{C\epsilon}, \quad C \geq 1. \end{displaymath}
Now, first, choose $C_{2}' \geq 1$ so large that $\nu(\{e : I_{s}(\pi_{e\sharp}\mu_{B}\}) \geq \delta^{-C_{2}'\epsilon}\}) \leq \delta^{C_{1}\epsilon}\nu(S^{1})$, and let $E' := \{e : I_{s}(\pi_{e\sharp}\mu_{B}) \leq \delta^{-C_{2}'\epsilon}\}$. One evidently needs $\geq (1 - \delta^{C_{1}\epsilon})|E_{B}|$ arcs of the form $B(e,\delta^{1/2}) \cap S^{1}$, $e \in E_{B}$ to cover $E'$, and this gives rise to a subset $E \subset E_{B}$ with $|E| \geq (1 - \delta^{C_{1}\epsilon})|E_{B}|$. I claim that these are of desired "good" vectors. 

For every $e \in E$, by definition, there exists a vector $e' \in B(e,\delta^{1/2}) \cap S^{1}$ with 
\begin{equation}\label{form20} \iint \frac{d\mu x \, d\mu y}{|\pi_{e'}(x) - \pi_{e'}(y)|^{s}} = I_{s}(\pi_{e'\sharp}\mu) \lessapprox \delta^{-C_{2}'\epsilon}. \end{equation}
Now, if $P_{B}' \subset P$ is a subset of cardinality $|P_{B}'| \geq \delta^{C_{0}\epsilon}|P_{B}|$ (as in the statement of the lemma), then the probability measure $\mu'$, defined in the obvious way by restricting and re-normalising $\mu_{B}$ to the subset $\delta^{-1/2}P_{B}'$, still satisfies \eqref{form20} with the $\lessapprox$-parameters depending on $C_{0}$. It follows that $\calH^{s}_{\infty}(\pi_{e'}(\spt \mu')) \gtrapprox 1$ (with similar dependence on $C_{0}$), and hence $\pi_{e'}(\spt \mu')$ contains a $(\delta^{1/2},s,C_{2})$-set of cardinality $\geq \delta^{C_{2}\epsilon -s/2}$ by Proposition \ref{deltasSet}, if $C_{2} \geq 1$ is large enough (depending on $C_{0}$ and $C_{2}'$, which just depends on $C_{1}$). Since $\pi_{e'}(\spt \mu')$ is contained in the $\delta^{1/2}$-neighbourhood of $\pi_{e'}(\delta^{-1/2}P_{B}')$, the same conclusion holds for $\pi_{e'}(\delta^{-1/2}P_{B}')$. Finally, using $|e' - e| \leq \delta^{1/2}$, the conclusion remains valid for $\pi_{e}(\delta^{-1/2}P_{B}')$, and thus, rescaling by $\delta^{1/2}$, the projection $\pi_{e}(P_{B}')$ contains a $(\delta,s)$-set of cardinality $\geq \delta^{C_{2}\epsilon-s/2}$ for every $e \in E$. \end{proof}

Fix some constants $C_{0},C_{1} \geq 1$ and let $C_{2}$ be specified by the lemma (the constant $C_{0}$ will be fixed far below, wheres $C_{1}$ will be specified momentarily). I define $\calT_{B}' \subset \calT_{B}''$ by selecting the $\geq (1 - \delta^{C_{1}\epsilon})|\calT_{B}''|$ tubes indicated by the lemma. Then, if $C_{1} \geq 1$ is large enough, \eqref{form5+} continues to hold for every $q \in P_{B}$, and with $\calT_{B}''$ replaced by $\calT_{B}'$:
\begin{equation}\label{form5++} |\calT_{B}' \cap \calT_{\delta^{1/2}}^{B}(p)| \geq |\calT_{B}'' \cap \calT_{\delta^{1/2}}(p)| - \delta^{C_{1}\epsilon}|\calT_{B}''| \approx \delta^{-s/2}, \quad p \in P_{B}. \end{equation}
Note that the parameters in the "$\approx$"-notation here do not depend on $C_{1}$, assuming that $C_{1}$ is large enough. This completes the first refinement of $\calT_{B}''$: roughly speaking, the conclusion was that "without loss of generality", the sets $P_{B}$, $B \in \calB$, can be assumed to have large projections in every direction perpendicular to the tubes meeting $B$. The second refinement (from $\calT_{B}'$ to $\calT_{B}$) is concerned with the distribution of the balls $B \in \calB$ meeting a fixed $\delta^{1/2}$-tube $T \in \calT_{\delta^{1/2}}$. Roughly speaking, I claim that "without loss of generality", every tube $T \in \calT_{\delta^{1/2}}$ only meets a $(1 - s)$-dimensional family of balls $B \in \calB$.  

To formalise such thoughts, write
\begin{equation}\label{PDeltaHalf} P_{\delta^{1/2}} := \{p_{B} : B \in \calB\}. \end{equation}
Since the balls $B \in \calB$ were assumed $\delta^{1/2}$-separated (see \eqref{distB}), it follows that $P_{\delta^{1/2}}$ is a $(\delta^{1/2},1)$-set of cardinality $|P_{\delta^{1/2}}| = |\calB| \approx \delta^{-1/2}$.  Consider the following inequality (see explanations below):
\begin{equation}\label{form6} \sum_{T \in \calT_{\delta^{1/2}}} \mathop{\sum_{B,B' \in \calB}}_{B \neq B'}  \frac{\chi_{\calT_{B}' \cap \calT_{B'}'}(T)}{|p_{B} - p_{B'}|^{1 - s}} = \mathop{\sum_{B,B' \in \calB}}_{B \neq B'} \frac{|\calT_{B}' \cap \calT_{B'}'|}{|p_{B} - p_{B'}|^{1 - s}} \lessapprox \mathop{\sum_{B,B' \in \calB}}_{B \neq B'} \frac{1}{|p_{B} - p_{B'}|} \lessapprox \delta^{-1}. \end{equation} 
The first "$\lessapprox$"-inequality uses the fact that $\calT_{B}'$ (or $\calT_{B'}'$) is a $(\delta^{1/2},s)$-set of tubes: since $T \in \calT_{B}' \cap \calT_{B'}'$ implies that $p_{B},p_{B'} \in T$ (recall \eqref{pBAndT}), this can only hold for $\lessapprox 1/|p_{B} - p_{B'}|^{s}$ choices of $T \in \calT_{B}'$ (or $T \in \calT_{B'}'$). The second "$\lessapprox$"-inequality in \eqref{form6} follows simply from the fact that $P_{\delta^{1/2}}$ is a $(\delta^{1/2},1)$-set.

Fix a large constant $C_{3} \geq 1$. It follows from \eqref{form6} and Chebyshev's inequality that
\begin{equation}\label{form21} \mathop{\sum_{B,B' \in \calB}}_{B \neq B'}  \frac{\chi_{\calT_{B}' \cap \calT_{B'}'}(T)}{|p_{B} - p_{B'}|^{1 - s}} \geq \delta^{-C_{3}\epsilon + s - 1} \end{equation}
can only hold for $\lessapprox \delta^{C_{3}\epsilon - s}$ tubes $T \in \calT_{\delta^{1/2}}$. Recalling that $|\calT_{\delta^{1/2}}| \approx \delta^{-s}$ by \eqref{form3+}, this roughly says that the tubes satisfying \eqref{form21} are exceptional. I need a the following slightly more accurate statement: for half of the balls $B \in \calB$, only a tiny fraction of the tubes in $T \in \calT_{B}$ can satisfy \eqref{form21}, if $C_{3} \geq 1$ was chosen large enough. Indeed, recall the constant $C_{1}$ from the previous page and assume that, for a certain $C_{3} \geq 1$, it holds that $\geq \delta^{C_{1}\epsilon}|\calT_{B}'| \approx \delta^{C_{1}\epsilon - s/2}$ tubes in $\calT_{B}'$ satisfy \eqref{form21} for $B \in \calB' \subset \calB$, where $|\calB'| \geq |\calB|/2$. Then Lemma \ref{2sBound} applies at scale $\delta^{1/2}$, and with the $(\delta^{1/2},1)$-set $\calB'$, and implies that the total number of tubes in $\calT_{\delta^{1/2}}$ satisfying \eqref{form21} is $\gtrapprox \delta^{-s}$, where the implicit parameters depend on $C_{1}$, but clearly not on $C_{3}$. Comparing this with the upper bound $\lessapprox \delta^{C_{3}\epsilon - s}$ gives an upper bound for $C_{3}$, which depends on $C_{1}$. 

Hence, assuming that $C_{3} \geq 1$ is large enough, the converse of \eqref{form21} holds for all $B \in \calB'$, and for $\geq (1 - \delta^{C_{1}\epsilon})|\calT_{B}'|$ tubes in $\calT_{B}'$. These tubes will be denoted by $\calT_{B}$. And once more, if $C_{1} \geq 1$ is large enough, the analogue of \eqref{form5++} continues to hold for $\calT_{B}$:
\begin{equation}\label{form5+++} |\calT_{B} \cap \calT_{\delta^{1/2}}^{B}(p)| \approx \delta^{-s/2}, \quad p \in P_{B}, \: B \in \calB'. \end{equation}
Again, the parameters in the "$\approx$"-notation do not, in fact, depend on $C_{1}$, assuming that $C_{1}$ is large enough. For each tube $T \in \calT_{B}$, I further observe that the number $M_{T} := |\{B \in \calB' : T \in \calT_{B}\}|$ satisfies
\begin{equation}\label{form7} M_{T} \lessapprox \delta^{(s - 1)/2}, \end{equation}
where the implicit constants depend on $C_{3}$. Indeed, the failure of \eqref{form7} (say: $M_{T} \geq \delta^{-(2C_{3}\epsilon + s - 1)/2}$) would imply that there are far more than $\delta^{s - 1}$ pairs $B,B' \in \calB' \subset \calB$ such that $T \in \calT_{B} \cap \calT_{B'} \subset \calT_{B}' \cap \calT_{B'}'$, which would violate \eqref{form21} (since $|p_{B} - p_{B'}| \lesssim 1$ for all pairs $B,B' \in \calB$). 

\subsection{Considerations at scale $\delta$} From now on, only the points in $P_{B}$, $B \in \calB'$, play any role in the proof. Set
\begin{displaymath} P' := \bigcup_{B \in \calB'} P_{B}, \end{displaymath}
which is a $(\delta,1)$-set of cardinality $\approx \delta^{-1}$, since $|\calB'| \approx \delta^{-1/2}$ and $|P_{B}| \approx \delta^{-1/2}$ for $B \in \calB$ by the definition of $P_{B}$ (just above \eqref{form5}). Fix $p \in P'$ and $B \in \calB'$ such that $p \in P_{B}$. For every 
\begin{displaymath} T \in \calT_{B} \cap \calT^{B}_{\delta^{1/2}}(p), \end{displaymath}
define $\calT_{p}'$ to consist of all the $\delta$-tubes $\calT_{p}$, which are $\delta$-children of $T$. By definition of $\calT_{\delta^{1/2}}^{B}(p)$ (see \eqref{form22}) and \eqref{form5+++}), the resulting subset $\calT_{p}' \subset \calT_{p}$ is a $(\delta,s)$-set of tubes containing $p$, and with 
\begin{equation}\label{calTq} |\calT_{p}'| \approx \delta^{-s}. \end{equation}
Thus, the set $P' \subset P$ and the families $\calT_{p}' \subset \calT_{p}$, $p \in P'$, satisfy precisely the same hypotheses as the original families $P$ and $\calT_{p}$ in Theorem \ref{main}. So, for notational, convenience, I re-define $P := P'$, $\calT_{p} := \calT_{p}'$, and $\calB := \calB'$. As before (in Remark \ref{calTcalP}), I continue to assume, without loss of generality, that
\begin{displaymath} \calT := \bigcup_{p \in P} \calT_{p}. \end{displaymath}
I also re-define $\calT_{\delta^{1/2}}$ to be the union of the families $\calT_{B}$, $B \in \calB$. Note that, with this definition of $\calT_{\delta^{1/2}}$, one has \eqref{form7} for all tubes in $\calT_{\delta^{1/2}}$. 

Compared with the original families $\calT_{p}$, the new families $\calT_{p}$ now enjoy additional regularity properties, which will be useful during the remainder of the proof. To exploit these, I record the following observation:
\begin{lemma}\label{monotonicity} Assume that $p \in B \in \calB$ and $T \in \calT_{p}$. Then the $\delta^{1/2}$-parent of $T$ belongs to $\calT_{B}$. In particular, the $\delta^{1/2}$-parents of the tubes in $\calT$ belong to $\calT_{\delta^{1/2}}$.
\end{lemma}
\begin{proof} This follows immediately from the construction of $\calT_{p}'$ -- which is now called $\calT_{p}$. \end{proof}

For $T \in \calT$, write
\begin{displaymath} N_{T} := |\{p \in P : T \in \calT_{p}\}|, \end{displaymath}
which is the analogue of the number $M_{T}$ at scale $\delta$. I make the following (rather familiar) claim: for at least half of the points $p \in P$, only a tiny fraction of the tubes in $\calT_{p}$ can \textbf{fail} to satisfy $N_{T} \lessapprox \delta^{s - 1}$. The proof is virtually the same as for the numbers $M_{T}$. One starts with the inequality
\begin{displaymath} \sum_{T \in \calT} \mathop{\sum_{p,q \in P}}_{p \neq q} \frac{\chi_{\calT_{p} \cap \calT_{q}}(T)}{|p - q|^{1 - s}} \lessapprox \delta^{-2}, \end{displaymath} 
which is analogous to and proven in the same way as \eqref{form6}. Thus, only $\lessapprox \delta^{C_{4}\epsilon - 2s}$ tubes in $\calT_{p}$ can satisfy 
\begin{displaymath} \mathop{\sum_{p,q \in P}}_{p \neq q} \frac{\chi_{\calT_{p} \cap \calT_{q}}(T)}{|p - q|^{1 - s}} \geq \delta^{-C_{4}\epsilon + 2s - 2}. \end{displaymath}
Hence, using Lemma \ref{2sBound} as before, the inequality above can only hold for a tiny fraction (depending on $C_{4}$) of the tubes in $\calT_{p}$, for half of the points in $P$. This implies the statement about the numbers $N_{T}$. Now, as final refinement, I only keep the "good" half of the points in $P$, and for those $p \in P$, I re-define $\calT_{p}$ to consist of the tubes $T$ with 
\begin{equation}\label{form1} N_{T} \lessapprox \delta^{s - 1}. \end{equation}
If $C_{4}$ was large enough, the cardinality estimate \eqref{calTq} stays valid. Finally, if $\calT$ is re-defined as the union of the (remaining) tubes in $\calT_{p}$, $p \in P$, one may assume that \eqref{form1} holds uniformly for all $T \in \calT$. 

Fix $p \in P$ and $T \in \calT$. Recall that the pair $(p,T)$ is called an \emph{incidence}, if $T \in \calT_{p}$, and the collection of all incidences is denoted by $I(P,\calT) := \{(p,T) : T \in \calT_{p}\}$. Evidently
\begin{equation}\label{incidences2} |I(P,\calT)| = \sum_{p \in P} |\calT_{p}| \approx \delta^{-s - 1}. \end{equation}
By the uniform upper bound \eqref{form1}, any tube $T \in \calT$ can only be incident to $\lessapprox \delta^{s - 1}$ points in $P$. Since $|\calT| \lessapprox \delta^{-2s}$ by the main counter assumption \eqref{counterAss}, the estimate \eqref{incidences2} shows that there exist $\approx \delta^{-2s}$ tubes in $\calT$ with $N_{T} \approx \delta^{s - 1}$. These tubes will be called \emph{good} tubes, and they will be denoted by $\calT^{G}$. 
\begin{lemma}\label{goodTubes} Any fixed tube $T_{0} \in \calT_{\delta^{1/2}}$ can only have $\lessapprox \delta^{-s}$ children in $\calT^{G}$. 
\end{lemma}

\begin{proof} Write $\calT^{G}(T_{0}) := \{T \in \calT^{G} : T \subset T_{0}\}$, and let $I^{G}(T_{0})$ be the set of incidences
\begin{displaymath} I^{G}(T_{0}) := \{(p,T) \in P \times \calT : T \in \calT_{p} \cap \calT^{G}(T_{0})\}. \end{displaymath}
By the definition of good tubes, evidently
\begin{displaymath} |I^{G}(T_{0})| \gtrapprox |\calT^{G}(T_{0})|\delta^{s - 1}. \end{displaymath}
On the other hand, by Lemma \ref{monotonicity}, an incidence $(p,T) \in P \times \calT_{p}$ can only occur, if the $\delta^{1/2}$-parent of $T$ belongs to $\calT_{B_{p}}$ for the (unique) ball $B_{p} \in \calB$ containing $p$. For $T \subset T_{0}$, the $\delta^{1/2}$-parent is evidently $T_{0}$, so 
\begin{displaymath} T \subset T_{0} \text{ and } (p,T) \in \calT_{p} \quad \Longrightarrow \quad T_{0} \in \calT_{B_{p}}. \end{displaymath}
Now, recall from the estimate $M_{T_{0}} \lessapprox \delta^{(s - 1)/2}$ (see \eqref{form7}), that there are only $\lessapprox \delta^{(s - 1)/2}$ balls $B$ with $T_{0} \in \calT_{B}$. For every such a ball $B$, every point $p \in P_{B}$ can be incident to $\lessapprox \delta^{-s/2}$ tubes $T \in \calT_{p}$ with $T \subset T_{0}$ (for the simple reason that $T_{0}$ only contains $\lessapprox \delta^{-s/2}$ tubes in $\calT_{p}$ by Lemma \ref{auxLemma}). Recalling that $|P_{B}| \lessapprox \delta^{-1/2}$, this gives the upper bound
\begin{equation}\label{form23} |I^{G}(T_{0})| \leq \sum_{B : T_{0} \in \calT_{B}} |\{(p,T) \in B \times \calT : T \in \calT_{p} \cap \calT^{G}(T_{0})\}| \lessapprox \delta^{(s - 1)/2} \cdot \delta^{-1/2} \cdot \delta^{-s/2} = \delta^{-1}. \end{equation} 
Comparing with the lower bound for $|I^{G}(T_{0})|$ completes the proof. \end{proof}
To sum up the most recent observations, there are $\approx \delta^{-2s}$ good tubes, each one of which is contained in some tube of $\calT_{\delta^{1/2}}$, and each tube in $\calT_{\delta^{1/2}}$ can only contain $\lessapprox \delta^{-s}$ good tubes. By the main counter assumption \eqref{counterAss}, moreover, one has $|\calT_{\delta^{1/2}}| \lessapprox \delta^{-s}$, which finally implies that there exists a tube $T_{0} \in \calT$ with $|\calT^{G}(T_{0})| \approx \delta^{-s}$. For simplicity, write 
\begin{equation}\label{calT0} \calT_{0} := \calT^{G}(T_{0}). \end{equation} 

I now claim that there are $\approx \delta^{(s - 1)/2}$ balls $B \in \calB$, say $\calB_{0}$, such that $T_{0} \in \calT_{B}$ for $B \in \calB_{0}$, and such that in each ball $B \in \calB_{0}$ one finds $\approx \delta^{-1/2}$ points of $P_{B}$, say $P_{B}'$, with a near-maximal number of incidences with $\calT_{0}$, namely
\begin{equation}\label{form8} |\calT_{p} \cap \calT_{0}| \approx \delta^{-s/2}, \qquad p \in P_{B}', \: B \in \calB_{0}. \end{equation}
This follows directly from the proof of Lemma \ref{goodTubes}. First observe that $|I^{G}(T_{0})| \approx \delta^{-1}$, since $|\calT_{0}| \approx \delta^{-s}$. Next, have a look at the upper bound \eqref{form23}, and observe that if any part of the claim failed, the bound would be lower than $\delta^{-1}$. This establishes the claim. 

Now, since $T_{0} \in \calT_{\delta^{1/2}}$, the converse of \eqref{form21} holds for $T_{0}$ (recall the definition of $\calT_{B}$ next to \eqref{form5+++}, and recall that every tube in $\calT_{\delta^{1/2}}$ belongs to $\calT_{B}$ for some $B$):
\begin{displaymath} \mathop{\sum_{B,B \in \calB_{0}}}_{B \neq B'} \frac{1}{|p_{B} - p_{B'}|^{1 - s}} \lessapprox \delta^{s - 1}. \end{displaymath}
Using Chebyshev's inequality, this implies that a further subset $\calB_{0}' \subset \calB_{0}$ of cardinality $|\calB_{0}'| \approx \delta^{(s - 1)/2}$ satisfies
\begin{displaymath} \sum_{B' \in \calB_{0}} \frac{1}{|p_{B} - p_{B}'|^{1 - s}} \lessapprox \delta^{(s - 1)/2}, \quad B \in \calB_{0}'. \end{displaymath}
and then $\calB_{0}'$ is a $(\delta^{1/2},1 - s)$-set of cardinality $\approx \delta^{(s - 1)/2}$. Since the balls $\calB_{0}'$ satisfy precisely the same estimates as $\calB_{0}$, I will continue writing $\calB_{0} := \calB_{0}'$. For convenience, assume that $T_{0}$ is a vertical tube (that is, change coordinates so that this holds). Then the $y$-coordinates of the points $p_{B} \in B$, $B \in \calB_{0}$, form a $(\delta^{1/2},1 - s)$-set in $[-1,1]$. Denote these $y$-coordinates by $A_{1} := \{p_{B}^{y}$, $B \in \calB_{0}\}$.

\subsection{Quasi-product sets, and concluding the proof of Theorem \ref{main}}\label{quasiProduct} Now, recall (from above \eqref{form8}) the subsets $P_{B}' \subset P_{B}$, defined for $B \in \calB_{0}$. They have cardinality $|P_{B}'| \geq \delta^{C_{0}\epsilon}|P_{B}|$ for some constant $C_{0} \geq 1$. This is the constant with which one wants to apply Lemma \ref{projections}: since $T_{0} \in \calT_{B}'$ (recall \eqref{form5++}), the projection of $\pi(P_{B}')$ of $P_{B}'$ to the $x$-axis contains a $(\delta,s)$-set $\Delta_{B}$ of cardinality $|\Delta_{B}| \geq \delta^{C_{2}\epsilon-s/2}$.

Consider the "quasi-product set"
\begin{displaymath} F' := \bigcup_{p_{B}^{y} \in A_{1}} \Delta_{B} \times \{p_{B}^{y}\}. \end{displaymath}
Fix $(a,b) = (a,p_{B}^{y}) \in F'$, $B \in \calB_{0}$. Then $a = \pi(p)$ for some $p = p_{(a,b)} \in P_{B}'$, so that $|(a,b) - p| \leq 2\delta^{1/2}$. Recall that \eqref{form8} holds for $p$, and let $T \in \calT_{p} \cap \calT_{0}$. By elementary geometry, using $|(a,b) - p| \leq 2\delta^{1/2}$ and $p \in T \subset T_{0}$, the point $(a,b) \in F'$ is covered by $B(0,10) \cap T'$ for some dyadic $\delta$-tube $T'$ in the $C\delta$-neighbourhood of $T$. (By this, I mean that if $T = \calD(Q)$, then $T' = \calD(Q')$ for some dyadic $\delta$-square $Q'$ with $\dist(Q,Q') \leq C\delta$.) This is best explained by a picture, see Figure \ref{fig1}.
\begin{figure}[h!]
\begin{center}
\includegraphics[scale = 0.7]{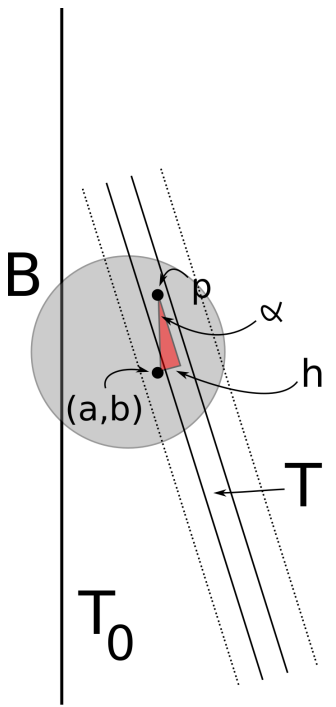}
\caption{The geometry of the tubes $T_{0},T$ and the points $(a,b)$ and $p = p_{(a,b)}$. Since $T \subset T_{0}$, the angle $\alpha$ is $\lesssim \delta^{1/2}$. Since $p,(a,b) \in B$, the distance between the points $p$ and $(a,b)$ is no greater than $2\delta^{1/2}$. Hence $h \lesssim \delta$, which ensures that $(a,b)$ is contained in $B(0,10) \cap T'$ for some dyadic $\delta$-tube $T'$ parallel to, and at distance $\lesssim \delta$, from $T \in \calT_{p}$.}\label{fig1}
\end{center}
\end{figure}

Now, for each $T \in \calT_{0}$, choose an ordinary $C\delta$-tube parallel to $T$, which covers $B(0,10) \cap T'$ for all the dyadic $\delta$-tubes $T'$ in the $C\delta$-neighbourhood of $T$. The collection of all ordinary $C\delta$-tubes so obtained is denoted by $\calT^{o}$ (here "$o$" stands for "ordinary"). Then 
\begin{displaymath} |\calT^{o}| \sim |\calT_{0}| \approx \delta^{-s}. \end{displaymath}
In particular, $\calT^{o}$ contains the ordinary $C\delta$-tubes produced from the dyadic $\delta$-tubes in $\calT_{p_{(a,b)}} \cap \calT_{0}$. By \eqref{form8} and the discussion above, this gives rise to a $(\delta,s)$-subset $\calT^{o}_{(a,b)} \subset \calT^{o}$ of ordinary $C\delta$-tubes of cardinality $|\calT^{o}_{(a,b)}| \approx \delta^{-s/2}$, with the property that
\begin{displaymath} (a,b) \in T, \qquad T \in \calT^{o}_{(a,b)}. \end{displaymath}

Finally, consider the following affine transformation of $F'$:
\begin{displaymath} F := \bigcup_{b \in A_{1}} A_{b} \times \{b\}, \end{displaymath}
where $A_{b} := \delta^{-1/2}\Delta_{B}$. Note that each $A_{b}$ is a $(\delta^{1/2},s)$-set, and recall that $A_{1}$ is a $(\delta^{1/2},1 - s)$-set. Clearly $F = \operatorname{Aff}(F')$, where $\operatorname{Aff}(x,y) = (\delta^{-1/2}x,y)$. Then $\calT' := \operatorname{Aff}(\calT^{o}) = \{\operatorname{Aff}(T) : T \in \calT^{o}\}$ is a family of ordinary $C'\delta^{1/2}$-tubes of cardinality $|\calT'| \sim |\calT| \approx \delta^{-s}$. Moreover, every point $x = (\delta^{-1/2}a,b) \in F$ is is contained in a $(\delta^{1/2},s)$-subset $\calT_{x}' \subset \calT'$ of ordinary $C'\delta^{1/2}$-tubes with $|\calT_{x}'| \approx \delta^{-s/2}$, namely $\calT_{x}' := \operatorname{Aff}(\calT_{(a,b)}^{o})$. The existence of $F$ and the families $\calT',\calT_{x}'$, now contradict the next proposition (at scale $\delta^{1/2}$, with $P = F$, $\calT = \calT'$ and $\tau = 1 - s > 0$). This completes the proof of Theorem \ref{main}.

\subsection{An incidence theorem for quasi-product sets} The wording "quasi-product set" is rather informal, and simply refers to sets $P$ of the form \eqref{P} below (if all the sets $A_{b}$ were the same, then $P$ would truly be a product set). On the last meters of the proof above, such a set, namely $F$, was constructed: it turned out that the points $x \in F$ were each incident to a large family $\calT_{x}'$ of not-too-concentrated tubes, and all the families $\calT_{x}'$ were subsets of a fixed small family $\calT'$. The next, and final, proposition shows that this is simply not possible. 

\begin{proposition}\label{productProp} Given $0 < s < 1$ and $\tau > 0$, there exists a number $\epsilon = \epsilon(s,\tau) > 0$ such that the following holds. Let $B \subset [0,1]$ be a $(\delta,\tau,\delta^{-\epsilon})$-set of cardinality $|B| \gtrsim \delta^{-\tau + \epsilon}$, and for each $b \in B$, assume that $A_{b} \subset [0,1]$ is a $(\delta,s,\delta^{-\epsilon})$-set of cardinality $|A_{b}| \gtrsim\delta^{-s + \epsilon}$. Consider the $(\delta,s + \tau,\delta^{-2\epsilon})$-set
\begin{equation}\label{P} P := \bigcup_{b \in B} A_{b} \times \{b\}. \end{equation}
Assume that $\calT$ is a collection of (ordinary) $\delta$-tubes such every family $\{T \in \calT : p \in T\}$, $p \in P$, contains in a $(\delta,s,\delta^{-\epsilon})$-set $\calT_{p} \subset \calT$ with $|\calT_{p}| \gtrsim \delta^{-s + \epsilon}$. Then $|\calT| \gtrsim \delta^{-2s - \epsilon}$. 
\end{proposition}

The proof of Proposition \ref{productProp} is, again, based on a counter assumption, namely $|\calT| \leq \delta^{-2s - \epsilon}$. For the remainder of the paper, the notations $\lessapprox$, $\gtrapprox$ and $\approx$, and the concept of $(\delta,t)$-set, are defined exactly as before, in Section \ref{mainCounterAss}, but now relative to the "$\epsilon$" in this counter assumption. Naturally, the implicit constants $C,C_{\epsilon}$ are now also allowed to depend on $\tau$, in addition to $s$. 

Before starting the proof of Proposition \ref{productProp} in earnest, I recall two standard results from additive combinatorics. The first is the \emph{Balog-Szemer\'edi-Gowers theorem}. The statement below is taken verbatim from p. 196 in \cite{Bo}. For a proof, see \cite{TV}, p. 267.
\begin{thm}[Balog-Szemer\'edi-Gowers]\label{BSG} There exists an absolute constant $C \geq 1$ such that the following holds. Let $A,B \subset \R$ be finite sets, and assume that $G \subset A \times B$ is a set of pairs such that
\begin{displaymath} |G| \geq \frac{|A||B|}{K} \quad \text{and} \quad |\{x + y : (x,y) \in G\}| \leq K|A|^{1/2}|B|^{1/2} \end{displaymath}
for some $K > 1$. Then, there exist $A' \subset A$ and $B' \subset B$ satisfying
\begin{itemize}
\item $|A'| \geq K^{-C}|A|$, $|B'| \geq K^{-C}|B|$,
\item $|A' + B'| \leq K^{C}|A|^{1/2}|B|^{1/2}$, and
\item $|G \cap (A' \times B')| \geq K^{-C}|A||B|$. 
\end{itemize}
\end{thm}

The second auxiliary result is the \emph{Pl\"unnecke-Ruzsa inequality}, whose proof can also be found in \cite{TV}:
\begin{thm}[Pl\"unnecke-Ruzsa]\label{PR} Assume that $A,B \subset \R$ are finite sets such that
\begin{displaymath} |A + B| \leq C|A| \end{displaymath}
for some integer $C \geq 1$. Then
\begin{displaymath} |B^{m} \pm B^{n}| \leq C^{m + n}|A| \end{displaymath}
for all $m,n \in \N$.
\end{thm}

\begin{remark} Theorem \ref{PR} will be applied in the following form: if $A,B \subset \R$ are $\delta$-separated sets with $|A| \approx |B|$ and
\begin{displaymath} N(A + B,\delta) \lessapprox |A|, \end{displaymath}
then $N(B + B,\delta) \lessapprox |A|$. This statement follows easily from Theorem \ref{PR} by considering the sets $[A]_{\delta} = \{[a]_{\delta} : a \in A\} \subset \delta \Z$ and $[B]_{\delta} := \{[b]_{\delta} ; b \in B\} \subset \delta \Z$, where $[x]_{\delta} \in \delta \Z$ stands for the largest number $\delta n \in \delta \Z$ satisfying $\delta n \leq x$. Then the hypothesis $N(A + B,\delta) \lessapprox |A|$ implies that $|[A]_{\delta} + [B]_{\delta}| \lessapprox |[A]|_{\delta}$, so Theorem \ref{PR} can be applied.
\end{remark}

\begin{proof}[Proof of Proposition \ref{productProp}] I start by making three convenient extra assumptions, which are not difficult to arrange. First, every tube in $\calT$ meets only one point in each set $A_{b} \times \{b\}$ (that is, the tubes in $\calT$ are "roughly vertical"); this can be arranged by restricting attention to those tubes in each $\calT_{p}$, which form an angle $\gtrapprox 1$ with horizontal lines. By the $(\delta,s)$-set hypothesis, $\approx \delta^{-s}$ tubes remain in each $\calT_{p}$, and then one can re-define $\calT$ as the union of the reduced families $\calT_{p}$. In particular, now each tube in $\calT$ only intersects the lines $\R \times \{b\}$, $b \in B$, inside a single interval of length $\lessapprox \delta$. After this procedure, one can remove some points from each $A_{b} \times \{b\}$ so that the mutual separation exceeds the length of the intervals mentioned above; again, by the $(\delta,s)$-set hypothesis, this can be arranged to that $\approx \delta^{-s}$ points remain for every $b$. 

Second, 
\begin{displaymath} A_{b} \subset \delta \Z, \qquad b \in B. \end{displaymath}
This can be arranged by perturbing the points of $A_{b}$ by $\leq \delta$. The tubes in $\calT_{p}$ may no longer contain $p$, but the $100\delta$-neighbourhoods of the tubes in $\calT_{p}$ certainly do. These neighbourhoods can be covered by $200$ ordinary $\delta$-tubes, each, which gives rise to a new family $\calT'$ of ordinary $\delta$-tubes with $|\calT'| \sim |\calT|$. Then, one can prove the proposition for $\calT'$ instead of $\calT$.

 Third, if $b_{1},b_{2},b_{3} \in B$, and $T,T' \in \calT$ are two tubes both containing certain points $(a_{1},b_{1}) \in A_{b_{1}} \times \{b_{1}\}$ and $(a_{3},b_{3}) \in A_{b_{3}} \times \{b_{3}\}$, then $T \cup T'$ can only contain one point in $A_{b_{2}} \times \{b_{2}\}$. This is similar to the first reduction: it follows from the assumption $T \cup T'$ intersects $\R \times \{b_{2}\}$ inside a single interval of length $\lessapprox \delta$ (since $T \cup T'$ is contained in the $\delta$-neighbourhood of the line connecting $(a_{1},b_{1})$ to $(a_{3},b_{3})$, and both tubes were already assumed to be roughly vertical). Thus, if the separation of $A_{b_{2}}$ exceeds the length of any such interval, the claim is true. And this can, as before, be arranged by discarding a few points from $A_{2}$.

The proof starts in earnest now, and I make the counter assumption $|\calT| \leq \delta^{-2s - \epsilon}$, or, in short,
\begin{equation}\label{counterAssumption} |\calT| \lessapprox \delta^{-2s}. \end{equation} 

For $T \in \calT$, write
\begin{displaymath} N_{T} := |\{p \in P : T \in \calT_{p}\}|. \end{displaymath}
Then, by the first "convenient extra assumption" above, one has the uniform bound 
\begin{equation}\label{form11} N_{T} \leq |B| \lessapprox \delta^{-\tau}. \end{equation}
On the other hand
\begin{displaymath} \sum_{T \in \calT} N_{T} = \sum_{p \in P} |\calT_{p}| \gtrapprox |P|\delta^{-s} = \delta^{-2s - \tau}. \end{displaymath}
By the counter assumption \eqref{counterAssumption}, one sees that $N_{T} \approx \delta^{-\tau}$ for $\approx \delta^{-2s}$ tubes in $\calT$. Consequently,
\begin{equation}\label{form12} \sum_{T \in \calT} |\{(p,q) \in P \times P : p \sim_{T} q\}| \gtrapprox \delta^{-2s - 2\tau}, \end{equation}
where $p \sim_{T} q$, if and only if $p \neq q$ and $T \in \calT_{p} \cap \calT_{q}$. Write $p \sim q$, if $p \sim_{T} q$ for some $T \in \calT$, and $Q := \{(p,q) : p \sim q\}$. Then the left hand side of the inequality above can be re-written and estimated as
\begin{displaymath} \sum_{p \sim q} |\calT_{p} \cap \calT_{q}| \lessapprox \sum_{p \sim q} \frac{1}{|p - q|^{s}} \lesssim |Q|^{1/r_{1}} \left( \sum_{p \neq q} \frac{1}{|p - q|^{s + \tau}} \right)^{1/r_{2}} \lessapprox |Q|^{1/r_{1}}|P|^{2/r_{2}}. \end{displaymath} 
The first inequality follows from the $(\delta,s)$-set hypothesis of either $\calT_{p}$ or $\calT_{q}$ (as in \eqref{form6}). The numbers $r_{1},r_{2} > 1$ are dual exponents such that $sr_{2} = s + \tau$, and the last inequality follows from the fact that $P$ is a $(\delta,s + \tau)$-set with $|P| \approx \delta^{-s - \tau}$. From this and \eqref{form12}, one infers that
\begin{displaymath} |Q| \gtrapprox |P|^{2}. \end{displaymath}
In heuristic terms, this shows that the graph with vertex set $P$ and edge set $\{(p,q) : p \sim q\}$ has almost maximal connectivity. Since there are no "edges" $p \sim q$ with $p,q \in A_{b} \times \{b\}$ for any fixed $b \in B$, the inequality above implies
\begin{equation}\label{form130} \sum_{b_{1} \neq b_{2}} |\{(p,q) \in A^{b_{1}} \times A^{b_{2}} : p \sim q\}| = |Q| \gtrapprox |P|^{2} \approx \delta^{-2s - 2\tau}, \end{equation}
where $A^{b_{i}} = A_{b_{i}} \times \{b_{i}\}$. 

Let $b_{1},b_{2} \in B$, and assume that $(p,q) \in A^{b_{1}} \times A^{b_{2}}$ satisfy $p \sim q$. Then, by definition, there exists at least one tube $T_{(p,q)} \in \calT_{p} \cap \calT_{q}$. If there are several, pick exactly one and call it $T_{(p,q)}$. Also, make these choices so that $T_{(p,q)} = T_{(q,p)}$. Then, set 
\begin{displaymath} \calT^{b_{1},b_{2}} := \{T_{(p,q)} : (p,q) \in A^{b_{1}} \times A^{b_{2}} \text{ and } p \sim q\}, \end{displaymath}
and note that 
\begin{equation}\label{form25} \calT^{b_{1},b_{2}} = \calT^{b_{2},b_{1}}, \quad b_{1}, b_{2} \in B. \end{equation}
Now, if $(p_{1},q_{1}),(p_{2},q_{2}) \in A^{b_{1}} \times A^{b_{2}}$ are two distinct pairs, then the collections $\calT_{p_{1}} \cap \calT_{q_{1}}$ and $\calT_{p_{2}} \cap \calT_{q_{2}}$ are disjoint. Indeed, if $p_{1} \neq p_{2}$, say, then no tube can lie in both $\calT_{p_{1}}$ and $\calT_{p_{2}}$ (since this would imply $p_{1} \sim p_{2}$). This implies that $T_{(p_{1},q_{1})} \neq T_{(p_{2},q_{2})}$, and consequently $|\calT^{b_{1},b_{2}}| \geq |\{(p,q) \in A^{b_{1}} \times A^{b_{2}} : p \sim q\}|$. Hence
\begin{equation}\label{form15} \sum_{b_{1}, b_{2}} |\calT^{b_{1},b_{2}}| \gtrapprox |P|^{2} \approx \delta^{-2s - 2\tau} \end{equation}
by \eqref{form130}. Now, using the counter assumption $|\calT| \lessapprox \delta^{-2s}$, and recalling \eqref{form25}, one can perform the following estimate:
\begin{align*} \sum_{b_{1},b_{2},b_{3}} |\calT^{b_{1},b_{2}} \cap \calT^{b_{2},b_{3}}| & = \sum_{T \in \calT} \sum_{b_{2}} \sum_{b_{1}, b_{3}} \chi_{\calT^{b_{1},b_{2}}}(T)\chi_{\calT^{b_{2},b_{3}}}(T)\\
& = \sum_{T \in \calT} \sum_{b_{2}} \left(\sum_{b} \chi_{\calT^{b,b_{2}}}(T) \right)^{2}\\
& \geq \frac{1}{|\calT||B|} \left( \sum_{T \in \calT} \sum_{b,b_{2}} \chi_{\calT^{b,b_{2}}}(T) \right)^{2}\\
& \gtrapprox \frac{|P|^{4}}{|\calT||B|} \gtrapprox \delta^{-2s}|B|^{3}. \end{align*} 
Since evidently $|\calT^{b_{1},b_{2}} \cap \calT^{b_{2},b_{3}}| \leq |\calT| \lessapprox \delta^{-2s}$ for any triple $(b_{1},b_{2},b_{3})$, it follows that there exist $\approx |B|^{3}$ triples $(b_{1},b_{2},b_{3})$ with the property that 
\begin{equation}\label{form18} |\calT^{b_{1},b_{2}} \cap \calT^{b_{2},b_{3}}| \approx \delta^{-2s}. \end{equation}
As will be made precise in a moment, the condition $|\calT^{b_{1},b_{2}} \cap \calT^{b_{2},b_{3}}| \approx \delta^{-2s}$ roughly means that there are $\approx \delta^{-2s}$ points in $A_{b_{1}} \times A_{b_{3}}$ such that the projection of these points is small in a certain direction, determined by $b_{1},b_{2},b_{3}$. 

Consider a triple of distinct points $b_{1},b_{2},b_{3} \in B^{3}$ with $\calT^{b_{1},b_{2}} \cap \calT^{b_{2},b_{3}} \neq \emptyset$. Fix $T \in \calT^{b_{1},b_{2}} \cap \calT^{b_{2},b_{3}}$. Since $T \in \calT^{b_{1},b_{2}}$, one has $T \in \calT_{p_{1}} \cap \calT_{p_{2}}$ for some unique pair of points 
\begin{displaymath} p_{1} = (a_{1},b_{1}) \in A^{b_{1}} \quad \text{and} \quad p_{2} = (a_{2},b_{2}) \in A^{b_{2}}. \end{displaymath}
Similarly, because $T \in \calT^{b_{2},b_{3}}$, there exists yet another unique point 
\begin{displaymath} p_{3} = (a_{3},b_{3}) \in A^{b_{3}} \end{displaymath}
such that $T \in \calT_{p_{2}} \cap \calT_{p_{3}}$. In particular, gathering all the pairs $(a_{1},a_{3}) \in A_{b_{1}} \times A_{b_{3}}$ obtained this way, one sees that the tubes $T \in \calT^{b_{1},b_{2}} \cap \calT^{b_{2},b_{3}}$ give rise to a subset 
\begin{displaymath} G_{b_{1},b_{2},b_{3}}' \subset A_{b_{1}} \times A_{b_{3}}. \end{displaymath}
of cardinality 
\begin{equation}\label{form14} |G_{b_{1},b_{2},b_{3}}'| = |\calT^{b_{1},b_{2}} \cap \calT^{b_{2},b_{3}}|. \end{equation}
To see the cardinality claim, one needs to check that distinct tubes $T,T' \in \calT^{b_{1},b_{2}} \cap \calT^{b_{2},b_{3}}$ give rise to distinct pairs in $(a_{1},a_{3}),(a_{1}',a_{3}') \in A_{b_{1}} \times A_{b_{3}}$. For $T$ and $T'$, let $p_{1},p_{2},p_{3}$ and $p_{1}',p_{2}',p_{3}'$ be the unique points above. Suppose, for contradiction, that $a_{1} = a_{1}'$ and $a_{3} = a_{3}'$, which means that $p_{1} = p_{1}'$ and $p_{3} = p_{3}'$. Then $T,T' \in \calT_{p_{1}} \cap \calT_{p_{2}}$ and $T,T' \in \calT_{p_{1}} \cap \calT_{p_{2}'}$. This implies that $p_{2} \neq p_{2}'$, since otherwise two tubes in $T,T' \in \calT_{p_{1}} \cap \calT_{p_{2}}$ would have been chosen to $\calT^{b_{1},b_{2}}$ contrary to the construction. But then $T,T'$ are tubes both containing the points $p_{1} \in A^{b_{1}}$ and $p_{3} \in A^{b_{3}}$ such that the union $T \cup T'$ contains two distinct points $p_{2},p_{2}' \in A^{b_{2}}$. This contradicts the third "convenient extra assumption" made at the beginning of the proof, and establishes \eqref{form14}.

From now on, restrict attention to triples $(b_{1},b_{2},b_{3}) \in B^{3}$ such that
\begin{equation}\label{separation} \min_{i \neq j} |b_{i} - b_{j}| \approx 1. \end{equation}
Since the set of triples satisfying $\min |b_{i} - b_{j}| \leq \delta^{C\epsilon}$ for $C \geq 1$ has cardinality no larger than $\lessapprox \delta^{C\tau \epsilon - \tau}|B|^{2} \lessapprox \delta^{C \tau \epsilon}|B|^{3}$ (using the $(\delta,\tau)$-set hypothesis of $B$), a large enough choice of $C$, depending on $\tau$, guarantees that $|\calT^{b_{1},b_{2}} \cap \calT^{b_{2},b_{3}}| \approx \delta^{-2s}$ holds for $\approx |B|^{3}$ triples satisfying \eqref{separation}. Fix one such triple, and consider a pair $(a_{1},a_{3}) \in G_{b_{1},b_{2},b_{3}}'$. Recall how such points arise, and the notation for $p_{1},p_{2},p_{3}$. Let 
\begin{displaymath} L = \left\{x = \frac{a_{3} - a_{1}}{b_{3} - b_{1}}y + \frac{a_{1}b_{3} - a_{3}b_{1}}{b_{3} - b_{1}} : y \in \R\right\} \end{displaymath}
be the line spanned by $p_{1}$ and $p_{3}$; then, since $p_{1},p_{2},p_{3}$ all lie in the common $\delta$-tube $T \in \calT$, the line $L$ passes at distance $\lesssim \delta$ from $p_{2} = (a_{2},b_{2}) \in A_{b_{2}} \times \{b_{2}\}$, which implies
\begin{displaymath} \left| \frac{a_{3}(b_{2} - b_{1}) + a_{1}(b_{3} - b_{2})}{b_{3} - b_{1}} - a_{2} \right| \lessapprox \delta, \end{displaymath}
using the fact that the tubes in $\calT$ are nearly vertical. Recalling \eqref{separation}, this further implies that
\begin{displaymath} \left| \left(a_{1} + \frac{b_{2} - b_{1}}{b_{3} - b_{2}} a_{3} \right) - \frac{b_{3} - b_{1}}{b_{3} - b_{2}}a_{2} \right| \lessapprox \delta. \end{displaymath}
Consequently, if $\pi_{b_{1},b_{2},b_{3}}$ stands for the projection-like mapping
\begin{equation}\label{pi} \pi_{b_{1},b_{2},b_{3}}(x,y) = x + \frac{b_{2} - b_{1}}{b_{3} - b_{2}}y, \end{equation}
then $\pi_{b_{1},b_{2},b_{3}}(G_{b_{1},b_{2},b_{3}}')$ is contained in the $\lessapprox \delta$-neighbourhood of 
\begin{displaymath} \frac{b_{3} - b_{1}}{b_{3} - b_{2}}A_{b_{2}}. \end{displaymath}
Observing that $N([(b_{3} - b_{1})/(b_{3} - b_{2})]A_{b_{2}},\delta) \lessapprox \delta^{-s}$ by \eqref{separation}, it follows that
\begin{equation}\label{form11} N(\pi_{b_{1},b_{2},b_{3}}(G_{b_{1},b_{2},b_{3}}'),\delta) \lessapprox \delta^{-s}. \end{equation}

This holds for any triple $(b_{1},b_{2},b_{3}) \in B^{3}$ satisfying \eqref{separation} by definition of $G_{b_{1},b_{2},b_{3}}'$, but the information is most useful, if $|G_{b_{1},b_{2},b_{3}}'| \approx \delta^{-2s} \approx |A_{b_{1}} \times A_{b_{3}}|$, which holds for $\approx |B|^{3}$ triples (recall \eqref{form14} and \eqref{form18}). Write
\begin{displaymath} F_{b_{1},b_{2},b_{3}} := \left\{\left(a_{1},\left[\frac{b_{2} - b_{1}}{b_{3} - b_{2}}a_{3}\right]_{\delta} \right) : (a_{1},a_{3}) \in G'_{b_{1},b_{2},b_{3}}\right\} \subset A_{b_{1}} \times \left[\frac{b_{2} - b_{1}}{b_{3} - b_{2}}A_{b_{3}}\right]_{\delta}. \end{displaymath}
Recall that $[r]_{\delta}$ stands for the largest number of the form $\delta n$, $n \in \Z$, with $\delta n \leq r$, and $[R]_{\delta} := \{[r]_{\delta} : r \in R\}$. It follows easily from \eqref{form11} (and recalling $A_{b_{1}} \subset \delta \Z$) that
\begin{displaymath} |\{t_{1} + t_{2} : (t_{1},t_{2}) \in F_{b_{1},b_{2},b_{3}} \}| \lessapprox \delta^{-s}. \end{displaymath}
Moreover, since $|(b_{2} - b_{1})/(b_{3} - b_{2})| \approx 1$ for every triple $(b_{1},b_{2},b_{3})$ satisfying \eqref{separation}, it follows that $|F_{b_{1},b_{2},b_{3}}| \approx \delta^{-2s}$ whenever \eqref{separation} holds and $|G_{b_{1},b_{2},b_{3}}'| \approx \delta^{-2s}$. For such a \emph{good} triple $(b_{1},b_{2},b_{3})$, the Balog-Szemer\'edi-Gowers theorem, Theorem \ref{BSG}, implies that there exist subsets
\begin{displaymath} D^{1}_{b_{1},b_{2},b_{3}} \subset A_{b_{1}} \quad \text{and} \quad \tilde{D}^{2}_{b_{1},b_{2},b_{3}} \subset \left[\frac{b_{2} - b_{1}}{b_{3} - b_{2}}A_{b_{3}}\right]_{\delta} \end{displaymath}
such that $|D^{1}_{b_{1},b_{2},b_{3}}|,|\tilde{D}^{2}_{b_{1},b_{2},b_{3}}| \approx \delta^{-s}$,
\begin{equation}\label{form17} |(D^{1}_{b_{1},b_{2},b_{3}} \times \tilde{D}^{2}_{b_{1},b_{2},b_{3}}) \cap F_{b_{1},b_{2},b_{3}}| \approx \delta^{-2s} \end{equation}
and
\begin{equation}\label{form118} |D^{1}_{b_{1},b_{2},b_{3}} + \tilde{D}^{2}_{b_{1},b_{2},b_{3}}| \lessapprox \delta^{-s}. \end{equation}
Let
\begin{displaymath} D_{b_{1},b_{2},b_{3}}^{2} := \left\{a_{3} \in A_{b_{3}} : \left[\frac{b_{2} - b_{1}}{b_{3} - b_{2}}a_{3} \right]_{\delta} \in \tilde{D}^{2}_{b_{1},b_{2},b_{3}} \right\}. \end{displaymath}
It then follows from the definition of $F_{b_{1},b_{2},b_{3}}$ and \eqref{form17} that
\begin{equation}\label{form112} |G_{b_{1},b_{2},b_{3}}| := |(D^{1}_{b_{1},b_{2},b_{3}} \times D^{2}_{b_{1},b_{2},b_{3}}) \cap G_{b_{1},b_{2},b_{3}}'| \approx \delta^{-2s}. \end{equation}
for a good triple $(b_{1},b_{2},b_{3})$. Moreover, \eqref{form118} easily implies that
\begin{equation}\label{form13} N_{1} := N\left(D^{1}_{b_{1},b_{2},b_{3}} + \frac{b_{2} - b_{1}}{b_{3} - b_{2}}D^{2}_{b_{1},b_{2},b_{3}},\delta\right) \lessapprox \delta^{-s}. \end{equation}
Finally, combining \eqref{form13} with the Pl\"unnecke-Ruzsa inequality, Theorem \ref{PR}, gives
\begin{equation}\label{form16} N_{2} := N(D^{2}_{b_{1},b_{2},b_{3}} + D^{2}_{b_{1},b_{2},b_{3}},\delta) \approx N\left(\frac{b_{2} - b_{1}}{b_{3} - b_{2}}D^{2}_{b_{1},b_{2},b_{3}} + \frac{b_{2} - b_{1}}{b_{3} - b_{2}}D^{2}_{b_{1},b_{2},b_{3}},\delta \right) \lessapprox \delta^{-s} \end{equation}
for any good triple $(b_{1},b_{2},b_{3})$. Since there are $\approx |B|^{3}$ good triples $(b_{1},b_{2},b_{3})$, one can find $b_{1},b_{3}$ such that \eqref{form112}--\eqref{form16} hold for $\approx |B|$ choices of $b_{2}$ (and so that $(b_{1},b_{2},b_{3})$ remains a good triple). Fix such $b_{1},b_{3} \in B$. Then, a simple Cauchy-Schwarz argument (similar to the one before \eqref{form5}) shows that $|G_{b_{1},b_{2},b_{3}} \cap G_{b_{1},b_{2}',b_{3}}| \approx \delta^{-2s}$ for $\approx |B|^{2}$ pairs $(b_{2},b_{2}')$, with both $(b_{1},b_{2},b_{3})$ and $(b_{1},b_{2}',b_{3})$ being good triples. Now, one can finally fix $b_{2} \in B$ such that $(b_{1},b_{2},b_{3})$ is a good triple, and
\begin{equation}\label{form114} |G_{b}| := |G_{b_{1},b_{2},b_{3}} \cap G_{b_{1},b,b_{3}}| \approx \delta^{-2s} \end{equation}
for $\approx |B|$ choices of $b \in B$ such that $(b_{1},b,b_{3})$ is a good triple. I denote the set of $b \in B$ satisfying these conditions by $B_{0}$. With \eqref{form13} in mind, write 
\begin{displaymath} c_{b} := \frac{b - b_{1}}{b_{3} - b}, \qquad b \in B_{0}, \end{displaymath}
and abbreviate $c := c_{b_{1}}$ (note that $|c|,|c_{b}| \approx 1$ for all $b \in B_{0}$ by \eqref{separation}). Also, write
\begin{displaymath} D^{1} := D_{b_{1},b_{2},b_{3}}^{1}(\delta) \quad \text{and} \quad D^{2} := D^{2}_{b_{1},b_{2},b_{3}}(\delta), \end{displaymath}
where $R(\delta)$ stands for the $\delta$-neighbourhood of $R \subset \R^{d}$. To complete the proof, I repeat an argument of Bourgain (see p. 219 in \cite{Bo}). Assume for a moment that $x \in cD^{2} \times D^{2} \subset \R^{2}$ and $b \in B_{0}$. Then $\chi_{-G_{b}(\delta) - y}(x) = 1$, whenever 
\begin{displaymath} y \in -G_{b}(\delta) - x \subset -(D^{1} \times D^{2}) - (cD^{2} \times D^{2}) = -(D^{1} + cD^{2}) \times -(D^{2} + D^{2}), \end{displaymath}
(the first inclusion uses \eqref{form112} and \eqref{form114}) and the Lebesgue measure of such choices $y$ is evidently $\calL^{2}(G_{b}(\delta))$. This gives the inequality
\begin{equation}\label{form27} \chi_{cD^{2} + c_{b}D^{2}} \leq \frac{1}{\calL^{2}(G_{b}(\delta))} \int_{-(D^{1} + cD^{2}) \times -(D^{2} + D^{2})} \chi_{\pi_{b_{1},b,b_{3}}(-G_{b}(\delta)) - \pi_{b_{1},b,b_{3}}(y)} \, dy, \quad b \in B_{0}, \end{equation}
by the definition of $\pi_{b_{1},b,b_{3}}$ (see \eqref{pi}). Indeed, if $t \in cD^{2} + c_{b}D^{2} = \pi_{b_{1},b,b_{3}}(cD^{2} \times D^{2})$, then $t = \pi_{b_{1},b,b_{3}}(x)$ for some $x \in cD^{2} \times D^{2}$. As discussed above,
\begin{displaymath} t = \pi_{b_{1},b,b_{3}}(x) \in \pi_{b_{1},b,b_{3}}(-G_{b}(\delta) - y) = \pi_{b_{1},b,b_{3}}(-G_{b}(\delta)) - \pi_{b_{1},b,b_{3}}(y), \end{displaymath}
whenever for $y \in -G_{b}(\delta) - x \subset -(D^{1} + cD^{2}) \times -(D^{2} + D^{2})$, and the set of such points $y$ has measure $\calL^{2}(G_{b}(\delta))$.

Finally, integrating inequality \eqref{form27} and recalling \eqref{form13}, \eqref{form16}, \eqref{form11} and \eqref{form114}, one obtains
\begin{equation}\label{conclusion} \calL^{1}(cD^{2} + c_{b}D^{2}) \lesssim \frac{(N_{1}\delta)(N_{2}\delta)}{\calL^{2}(G_{b}(\delta))} \calL^{1}(\pi_{b_{1},b,b_{2}}(G_{b})) \lessapprox \delta^{1 - s}, \quad b \in B_{0}. \end{equation}
However, $D^{2} \times D^{2}$ is the $\delta$-neighbourhood of a generalised $(\delta,2s)$-set in the plane, so Bourgain's discretized projection theorem, Theorem 5 in \cite{Bo}, can be applied with $\alpha := 2s < 2 =: d$. If $\mu_{1}$ is the normalised counting measure on the set $\{c_{b}/c : b \in B_{0}\}$, then it is not hard to check that $\mu_{1}$ satisfies assumption (0.14) from \cite{Bo} for any $\tau_{0} > 0$ and some $\kappa > 0$ depending only on $\tau$ (using the fact that $B_{0}$ is a $(\delta,\tau)$-set with $|B_{0}| \approx \delta^{-\tau}$; if preferred, this is even easier to check, if one first reduces $B_{0}$ slightly so that the the derivative of $b \mapsto c_{b}$ has absolute value $\approx 1$ uniformly for $b \in B_{0}$). Thus, the conclusion (0.19) of \cite{Bo} states that some $c_{b}/c$ with $b \in B_{0}$ should satisfy $\calL^{1}(D^{2} + (c_{b}/c)D^{2}) \geq \delta^{1 - s - \epsilon_{2}}$ for some constant $\epsilon_{2} > 0$ depending only on $s$ and $\tau$. Recalling that $|c| \approx 1$, this evidently violates \eqref{conclusion}. A contradiction is reached, and the proof of Proposition \ref{productProp} is complete.  \end{proof}

\end{document}